\documentclass[11pt,leqno,oneside,letterpaper]{amsart}
\usepackage[letterpaper,left=1.75truein, top=1truein]{geometry}
\usepackage{amsmath,amsfonts,amssymb,amscd,amsthm,amsbsy,epsf,graphics}

\usepackage{color}

\usepackage[colorlinks,linkcolor=blue,citecolor=blue]{hyperref}
\usepackage{epsfig}
\usepackage{graphicx}
\newtheorem{thm}{Theorem}[section]
\newtheorem{cor}[thm]{Corollary}
\newtheorem{lemma}[thm]{Lemma}
\newtheorem{prop}[thm]{Proposition}

\newtheorem{que}[thm]{Question}
\newtheorem{rem}[thm]{Remark}

\newcommand{\BQ}{\mathbb{Q}}
\theoremstyle{definition}

\newtheorem{remark}[thm]{Remark}

\def\qed{{\hspace{2mm}{\small $\diamondsuit$}}}

\newcommand{\BZ}{\mathbb{Z}}
\newcommand{\tr}{\operatorname{tr}}

\def\fX{\mathfrak X}
\def\fR{\mathfrak R}
\def\fA{\mathfrak A}
\def\fB{\mathfrak B}
\def\ve{\varepsilon}
\newtheoremstyle{cases}
  {12pt plus 6 pt}%       Space above
  {2pt}%       Space below
  {\bfseries}   %       Body font
  {}%          Indent amount (empty = no indent, \parindent = para indent)
  {\bfseries}% Thm head font
  {.}%         Punctuation after thm head
  {.5em}%      Space after thm head: " " = normal interword space;
  {}%          Thm head spec (can be left empty, meaning `normal')
\theoremstyle{cases}

\numberwithin{subcase}{case} \numberwithin{subsubcase}{subcase}
\numberwithin{equation}{subsection}
\def\sfrac#1#2{\kern.1em\raise.5ex\hbox{$#1$}
    \kern-.1em/\kern-.05em\lower.25ex\hbox{$#2$}}
\def\T{{\mathcal T}}
\def\wT{\widetilde{\T}}

\newcommand{\BC}{\mathbb{C}}
\def\BQ{{\mathbb Q}}

\def\be{  \begin{equation} }
\def\ee{  \end{equation} }

\def\sfrac#1#2{\kern.1em\raise.5ex\hbox{$#1$}
        \kern-.1em/\kern-.05em\lower.25ex\hbox{$#2$}}

\def\T{{\mathcal T}}

\def\G{{\Gamma}}
 \def\d{{\delta}}
 
 \def\e{{\epsilon}}
 \def\l{{\lambda}}
 
 \def\m{{\mu}}
 
 \def\o{{\omega}}
  
   \def\s{{\sigma}}
 
 \def\a{{\alpha}}
 
 \def\p{{\partial}}
 \def\r{{\rho}}
 \def\ra{{\rightarrow}}

 \def\lra{{\longrightarrow}}
 \def\g{{\gamma}}

 \def\c{{\mathbb C}}
 \def\z{{\mathbb Z}}
 \def\2{{\mathbb Z_2}}
 \def\q{{\mathbb Q}}
 \def\t{{\tau}}

 \def\sl2{{SL(2,\mathbb C)}}
 
 \def\qed{{\hspace{2mm}{\small $\diamondsuit$}}}

 \def\pf{{\noindent{\bf Proof.\hspace{2mm}}}}

 \def\sk{{{\mbox{\tiny K}}}}
 \def\sw{{{\mbox{\tiny W}}}}
 
 \def\su{{{\mbox{\tiny $U$}}}}

\def\scr{{{\mbox{\footnotesize  $\chi_\r$}}}}
\def\scrp{{{\mbox{\footnotesize  $\overline \chi_{\overline\r}$}}}}
\def\sckp{{{\mbox{\footnotesize  $\overline\chi_k'$}}}}
\def\sckpp{{{\mbox{\footnotesize  $\overline\chi_k''$}}}}
\def\sc{{{\mbox{\footnotesize  $\overline \chi$}}}}
\def\sz{{{\mbox{\footnotesize $\chi_0$}}}}
\def\szj{{{\mbox{\footnotesize $\chi_i$}}}}

\def\scz{{{\mbox{\footnotesize $\overline \chi_0$}}}}
\def\sck{{{\mbox{\footnotesize $\overline \chi_k$}}}}
\def\sce{{{\mbox{\footnotesize $\chi_{\e^*(\r)}$}}}}
\def\ak{{{\mbox{$A_\sk(\sm,\sl)$}}}}

\def\A{{\mathcal A}}

\def\T{{\mathcal T}}
\def\wT{\widetilde{\T}}
\def\P{{\mathcal P}}

\def\wA{\widetilde{\A}}
\def\sm{{{\mbox{\footnotesize  $\mathfrak{M}$}}}}
  \def\sl{{{\mbox{\footnotesize  $\mathfrak{L}$}}}}
 \def\fm{{\mathfrak{m}}}
\def\ft{{\mathfrak{t}}}
\def\fp{{\mathfrak{p}}}
\def\fs{{\mathfrak{s}}}

\def\spec{\mathrm{Spec}}

\input diagrams

\begin{document}

\title[Character varieties, AJ Conjecture]{Character varieties, A-polynomials, and the AJ Conjecture}

\author[Thang  T. Q. L\^e]{Thang  T. Q. L\^e}
\address{School of Mathematics, 686 Cherry Street,
 Georgia Tech, Atlanta, GA 30332, USA}
\email{letu@math.gatech.edu}
\author{Xingru Zhang}
\address{Department of Mathematics,
 University at Buffalo
Buffalo, NY 14260-2900, USA}
\email{xinzhang@buffalo.edu}

\thanks{T. L. is supported in part by National Science Foundation grant DMS-1406419. \\
2010 {\em Mathematics Classification:} Primary 57M25.\\
{\em Key words and phrases: character variety, $A$-polynomial, AJ conjecture.}}

\begin{abstract}
We establish some facts about the behavior of the rational-geometric subvariety of the $SL_2(\c)$ or $PSL_2(\c)$ character variety of
a hyperbolic knot manifold under the restriction map to
the  $SL_2(\c)$ or $PSL_2(\c)$ character variety of the boundary torus,
 and use the results to  get some properties about the A-polynomials
 and  to  prove the AJ conjecture for certain class of knots in $S^3$ including in particular
 any $2$-bridge knot over which the double branched cover of $S^3$ is a lens space of prime order.
\end{abstract}

\maketitle
\vspace{-.6cm}
\begin{center}

%\today
\end{center}

\section{Introduction}\label{intro}

For a finitely generated group $\G$, let $R(\G)$ denote the $SL_2(\c)$-representation variety of $\G$,
$X(\G)$ the $SL_2(\c)$-character variety of $\G$,
and $\tr: R(\G)\ra X(\G)$ the  map which sends a representation
$\r\in R(\G)$ to its character $\scr\in X(\G)$.
When $\G$ is the fundamental group of a connected manifold $W$, we
also write $R(W)$, $X(W)$ for $R(\pi_1(W))$, $X(\pi_1(W))$ respectively
and call them the $SL_2(\c)$-representation variety of $W$ and the $SL_2(\c)$-character variety of $W$.
The counterparts of these notions when the target group
$SL_2(\c)$ is replaced by $PSL_2(\c)$ are similarly defined and are
denoted by  $\overline R(\G)$, $\overline X(\G)$, $\overline \tr$,
$\overline \r$, $\scrp$,
$\overline R(W)$, $\overline X(W)$ respectively. We refer to \cite{CS} for basics
about $SL_2(\c)$-representation and character varieties and to \cite{BZ} in $PSL_2(\c)$ case.

In this paper, a {\em variety} $V$ is a closed complex affine algebraic set, i.e. a subset of $\BC^n$ which is the zero locus of a set of polynomials in $\BC[x_1,\dots,x_n]$.  If among the sets of polynomials which define
the same variety $V$ there is one whose elements all
have rational coefficients, we say that $V$ is {\em defined over $\BQ$}. Similarly
a regular map between two  varieties
is said to be {\em defined over $\mathbb Q$}  if the map is given by
a tuple of polynomials with coefficients in $\mathbb Q$.
Note that $R(\G)$, $X(\G)$, $\tr$,
$\overline R(\G)$, $\overline X(\G)$, $\overline \tr$
  are all defined over $\mathbb Q$.

In this paper irreducible varieties will be called $\c$-irreducible varieties.
Recall that a variety is {\em $\BC$-irreducible} if
it is not a union of two
proper subvarieties. Any variety $V$ can be presented as an irredundant union of $\BC$-irreducible subvarieties, each is called a {\em $\BC$-component} of $V$.
Similarly, a variety  defined over $\BQ$ is {\em $\BQ$-irreducible} if
it is not a union of two
proper subvarieties defined over $\BQ$. Any variety $V$ defined over $\BQ$ can be presented as an irredundant union of $\BQ$-irreducible subvarieties, each is called a {\em $\BQ$-component} of $V$.
In general a $\BQ$-component can be further decomposed into $\BC$-components.

If $\G_1$ and $\G_2$ are two finitely generated groups
and $h:\G_1\ra \G_2$ is a group homomorphism, we use
$h^*$ to denote the induced regular map from
 $R(\G_2)$, $X(\G_2)$, $\overline R(\G_2)$, or $\overline X(\G_2)$
 to $R(\G_1)$, $X(\G_1)$, $\overline R(\G_1)$ or $\overline X(\G_1)$
 respectively.
 Note that $h^*$ is defined over $\mathbb Q$.

Let $M$ be a knot manifold, i.e. $M$ is a connected compact orientable $3$-manifold
whose boundary $\p M$  is a torus.
Let $\iota^*$ be the regular map from
$R(M)$, $X(M)$, $\overline R(M)$ or $\overline X(M)$ to
$R(\p M)$, $X(\p M)$, $\overline R(\p M)$ or $\overline X(\p M)$
respectively, induced from the inclusion induced homomorphism
$\iota: \pi_1(\p M)\ra \pi_1(M)$.

We call a character $\scr$ (or $\scrp$)   reducible or irreducible or
discrete faithful or dihedral if the  corresponding
representation $\r$ (or $\overline \r$) has that property.

\def\rg{\mathrm {rg}}
\def\oX{\overline X}

\subsection{Rational-geometric subvariety}
Suppose $M$ is a  hyperbolic knot manifold, i.e.
 a knot manifold whose  interior has a complete
hyperbolic metric of finite volume. There are
precisely two discrete faithful characters in $\oX(M)$ (which follows from the Mostow-Prasad rigidity)
and there are precisely $2|H_1(M;\z_2)|$ discrete faithful characters in $X(M)$
(which follows from  a result of Thurston \cite[Proposition 3.1.1]{CS}).
The {\em rational-geometric subvariety}
$X^{\rg}(M)$ (respectively  $\overline X^{\rg}(M)$) is the union of $\q$-components of $X(M)$
 (respectively  $\overline X(M)$)  each of which
contains a  discrete faithful character.
 The number of $\q$-components  of $X^{\rg}(M)$ is at most $|H_1(M;\z_2)|$,
 and $\overline X^{\rg}(M )$ is $\q$-irreducible (which will be explained in
Section  \ref{sec:pf}), but it is not known how many $\c$-components    that $X^{\rg}(M)$ (respectively  $\overline X^{\rg}(M)$) can possibly
have.

In this paper we show

\begin{thm}\label{distinct}
Let $M$ be a hyperbolic knot manifold.
Let $\overline X_1,..., \overline X_l$ be the  $\BC$-components of $\overline X^{\rg}(M)$, and let
 $\overline Y_j$ be the Zariski closure of $\iota^*(\overline X_j)$ in $\overline X(\p M)$, $j=1,...,l$.
\newline
(1) For each $j$,  $\overline X_j$ is a curve.
\newline
(2) The regular map  $\iota^*: \overline X_j\ra \overline Y_j$ is a birational isomorphism for each $j=1,...,l$.
\newline
(3) If the two discrete faithful characters of
$\overline X(M)$ are contained in the same $\BC$-component   of $\overline X(M)$,  then the curves $\overline Y_j$, $j=1,...,l$,  are mutually distinct in $\overline X(\p M)$.
\end{thm}

In $SL_2(\c)$-setting
  we have a similar result but we need some restriction on the knot manifold.

\begin{thm}\label{distinct3}
Suppose that $M$ is a hyperbolic knot manifold which is the exterior of a knot
in a homology $3$-sphere.
Let $X_1,...,X_k$ be the  $\c$-components of $X^{\rg}(M)$, and let
 $Y_j$ be the Zariski closure of $\iota^*(X_j)$ in $X(\p M)$, $j=1,...,k$.
\newline
(1) For each $j$,  $X_j$ is a curve.
\newline
(2) The regular map  $\iota^*: X_j\ra Y_j$ is a birational isomorphism for each $j=1,...,k$.
\newline
(3) If   the two discrete faithful characters of
$\overline X(M)$ are contained in the same $\BC$-component of $\overline X(M)$,  then the curves $Y_j$, $j=1,...,k$,  are mutually distinct in $X(\p M)$.
\end{thm}

\begin{remark} Although the condition ``the two discrete faithful characters of
$\overline X(M)$ are contained in the same $\BC$-component of $\overline X(M)$" is hard to check, there is no known example of a hyperbolic knot exterior in $S^3$ for which this condition is not satisfied.
\end{remark}

We give two applications of Theorem \ref{distinct3}, one on
estimating degrees of A-polynomials and one on proving  the AJ conjecture  for a certain class of knots, which is the main motivation of this paper.

\subsection{$A$-polynomial}
When a knot manifold $M$ is the exterior of a knot $K$
in a homology $3$-sphere $W$, we denote the $A$-polynomial of $K$  in variables $\sm$ and $\sl$
by $A_{\sk, \sw}(\sm,\sl)$,  as defined in \cite{CCGLS}.
When $W=S^3$, we simply write
$A_{\sk}(\sm,\sl)$
for  $A_{\sk, S^3}(\sm,\sl)$.
Note that $A_{\sk, \sw}(\sm,\sl)\in \z[\sm,\sl]$ has no repeated factors
and always contains the factor $\sl-1$.
Let the {\em non-abelian $A$-polynomial} be defined by
$$\widehat A_{\sk, \sw}(\sm,\sl):=\frac{A_{\sk, \sw}(\sm,\sl)}{\sl-1}.$$

We call the maximum power of $\sm$ (respectively of $\sl$)  in
$\widehat A_{\sk, \sw}(\sm,\sl)$
the $\sm$-degree (respectively the $\sl$-degree) of $\widehat A_{\sk, \sw}(\sm,\sl)$.

\def\ve{\varepsilon}
\def\irr{\mathrm{irr}}
\newcommand\no[1]{}

When $M$ is a finite volume hyperbolic $3$-manifold, the trace
field of $M$  is defined
to be the field generated by the values of a discrete faithful character of
  $M$ over the base field $\q$. It is known that the trace field of $M$ is a number field, i.e.
  a finite degree extension of  $\q$, with the
    extension degree  at least two.

\begin{thm}\label{degree from trace field}
Suppose that $M$ is a hyperbolic knot manifold which is the exterior of a knot
$K$ in a homology $3$-sphere $W$. Let $d$ be the extension
degree of the trace field of $M$ over $\q$.
If the two discrete faithful characters of
$\overline X(M)$ are contained in the same $\BC$-component of $\overline X(M)$,
then both the $\sm$-degree and  the $\sl$-degree of
$\widehat A_{\sk, \sw}(\sm,\sl)$
are at least $d$.
In particular both the $\sm$-degree and  the $\sl$-degree of $\widehat A_{\sk, \sw}(\sm,\sl)$
are at least $2$.
\end{thm}

\def\irr{{\mathrm{irr}}}
\def\redu{{\mathrm{red}}}
\subsection{AJ conjecture}

Suppose $M$ is the exterior of a knot in a homology $3$-sphere.
All the reducible characters in $X(M)$ (resp. $\oX(M)$)
form a unique $\c$-component of $X(M)$ (resp. $\oX(M)$), which we denote by $X^\redu(M)$ (resp. $\oX^\redu(M)$).
We use $X^\irr(M)$ (resp. $\oX^\irr(M)$) to denote the union of the rest of the $\BC$-components of $X(M)$ (resp. $\oX(M)$).
We caution that our definition of $X^\irr(M)$ (resp. $\oX^\irr(M)$) may not be the exact complement of $X^\redu(M)$ (resp. $\oX^\redu(M)$) in $X(M)$ (resp. $\oX(M)$)
and it  may  still contain finitely many
reducible characters.
All $X^\redu(M), X^\irr(M), \oX^\redu(M), \oX^\irr(M)$
 are varieties  defined over $\mathbb Q$.

 \no{There is no known example of a knot $K$ in $S^3$ for which there is a component of $X(M)$ whose image under $\iota^*$ does not have dimension 1.}

For a knot $K$ in $S^3$, its recurrence polynomial $\a_\sk(t, \sm,\sl)\in \z[t, \sm,\sl]$
is derived from  the colored Jones polynomials of $K$, see \cite{G,GaLe,Le}.
The AJ-conjecture raised  in \cite{G} (see also \cite{FGL})  anticipates a  striking  relation between
the colored Jones polynomials of $K$ and the $A$-polynomial of $K$. It states
that for every knot $K\subset S^3$, $\a_\sk(1,\sm,\sl)$  is equal to the A-polynomial
$\ak$ of $K$,
up to a factor depending on $\sm$ only.
The following theorem generalizes \cite[Theorem 1]{LT}
and is the main result of this paper
(see Section \ref{section-AJ conjecture} for  detailed
definitions of terms mentioned here and for more background description).
\def\oX{\overline X}
\begin{thm}\label{AJ}
Let $K$ be a knot in $S^3$ whose exterior $M$ is hyperbolic.
 Suppose the following conditions are satisfied:
 \newline
 (1) $\oX^\irr(M)=\oX^{\rg}(M)$
 and  the two discrete faithful characters of
$\overline X(M)$ are contained in the same $\BC$-component of $\overline X(M)$,
 \newline
 (2) the $\sl$-degree of the recurrence polynomial $\a_\sk(t, \sm,\sl)$ of $K$ is larger than one,
 \newline
 (3)  the localized skein module  $\overline{\mathcal{S}}$ of $M$ is finitely generated.
 \newline
 Then the AJ-conjecture holds for $K$.
\end{thm}

In \cite[Theorem 1]{LT}, it is required   that $X^\irr(M)=X^{\rg}(M)$ and both  are $\c$-irreducible, which is obviously stronger than our condition (1) of Theorem~\ref{AJ}. In general, irreducibility over $\c$ is difficult to check.
We also remove the condition required in  \cite[Theorem 1]{LT} that the universal $SL_2$-character ring of $M$
is reduced.

It was known that condition (2) of Theorem~\ref{AJ} is satisfied by any nontrivial adequate knot
(in particular any nontrivial alternating knot) in $S^3$ (see \cite{Le}) and
condition (3) of Theorem~\ref{AJ}
is satisfied by all $2$-bridge knots
  (see \cite{Le}) and all pretzel knots of the form $(-2,3,2n+1)$ (see \cite{LT}).
Concerning condition (1) of Theorem \ref{AJ}, we have the following.

\begin{thm}\label{2-bridge}Let $K$ be a $2$-bridge knot in $S^3$ with a hyperbolic exterior $M$.
\newline (1) The two discrete faithful characters of
 $\overline X(M)$ are  contained in the same $\BC$-component of $\overline X(M)$,
 \newline
 (2) All the four discrete faithful $SL_2(\c)$-characters are  contained in the same $\BC$-component of $ X(M)$, and
  $X^{\rg}(M)$ is irreducible over $\q$.
 \end{thm}

Therefore  we have the following corollary which
 generalizes \cite[Theorem 2 (b)]{LT}.

 \begin{cor}\label{cor for 2-bridge} Let $K$ be a $2$-bridge knot in $S^3$ with a hyperbolic exterior $M$.
%(a)
 If $X^\irr(M)=X^{\rg}(M)$, then the AJ-conjecture holds for $K$.

%(b) If $X^\irr(M)$ is $\q$-irreducible, then $X^\rg(M) = X^\irr(M)$.
\end{cor}

Note that a two-bridge knot has hyperbolic exterior if and only if it is not a torus knot, and for all torus knots the AJ  conjecture is known to hold \cite{Hikami,Tran}.

Since $X^\rg(M) \subset X^\irr(M)$ and $X^\rg(M)$ is defined over $\q$, if $X^\irr(M)$ is $\q$-irreducible, then $X^\rg(M) = X^\irr(M)$.
For a two-bridge knot, the variety $X^\irr(M)$ is the zero locus of the  Riley polynomial which is a polynomial in two variable, see \cite{Riley}. Hence, we have the following.

 \begin{cor}\label{2-bridge 2} Let $K$ be a $2$-bridge knot in $S^3$. If $X^\irr(M)$ is $\q$-irreducible, or if the Riley polynomial of $K$ is irreducible over $\q$, then the AJ conjecture holds for $K$.
\end{cor}
\def\bb{\mathfrak{b}}
In \cite[Section A1]{LT} it was proved that the Riley polynomial of the two bridge knot $\bb(p,q)$ is $\q$-irreducible if $p$ is a prime. Here we use the notation of \cite{BuZ} for two bridge knots: $\bb(p,q)$ is the two bridge knot such that the double branched covering of $S^3$ along $\bb(p,q)$ is the lens space $L(p,q)$. Note that both $p,q$ are odd numbers, co-prime with each other, and $1 \le q \le p-2$, and $\bb(p,q)$ is hyperbolic if and only if $q\neq 1$. When $q=1$, $\bb(p,1)$ is a torus knot, and the AJ conjecture for it holds.
Thus we have

 \begin{cor}\label{2-bridge 3} The AJ conjecture holds for all two bridge knots $\bb(p,q)$ with odd prime $p$.
\end{cor}
\no{
Actually, a conjecture of xxxx states that the Riley polynomial of $\bb(p,q)$ is $\q$-reducible if and only if xxxx
A direct calculation shows that  all two-bridge knots $\bb(p,q)$ with $ p \le 500$, except for are xxxx.
}

{\bf Plan of paper}.
In \S2, we prove Theorems \ref{distinct}, \ref{distinct3} and \ref{degree from trace field}. The proof of Theorem \ref{distinct}   applies the theory of
 volumes of representations  developed in \cite{H}, \cite{CCGLS}, \cite{D},
 \cite{F}, plus  the consideration of the $Aut(\c)$-action on varieties.
 Theorem \ref{distinct3} follows quickly from
 Theorem \ref{distinct} under the consideration of
the $H_1(M; \z_2)$-action on $X(M)$. Theorem \ref{degree from trace field}
follows from Theorem \ref{distinct3} together with the fact observed in \cite{SZ}
that the $(Aut(\c)\times H_1(M; \z_2))$-orbit of a discrete faithful
$SL_2(\c)$-character of $X(M)$, which of course  is contained in $X^\rg(M)$, contains
at least $2d$ elements.
\S 2 also contains some related results, notably Theorem \ref{distinct2} which is
a refinement of Theorem \ref{distinct}, and Proposition \ref{bal-irred} which gives
a property of A-polynomial that will be applied in the proof of
Theorem \ref{AJ} in \S4 (see also Remark \ref{a refinement}).
To prove our main result Theorem \ref{AJ}
we need to first prepare some properties concerning the representation schemes and character schemes of  knot manifolds
in \S3. In \S4, we illustrate how the  approach of \cite{LT}
 can be applied to reduce  Theorem \ref{AJ}
to  Proposition \ref{fs is reduced}.
This proposition will then  be proved in \S5, where Theorem \ref{distinct3} and results
from \S3 are applied.
In last section, we prove Theorem \ref{2-bridge} which follows easily from
results in \cite[Section 5]{T} and a result of \cite{BZ3}, concerning dihedral characters.

{\bf Acknowledgements}. We would like to thank Stefano Francaviglia, Tomotada Ohtsuki, Adam Sikora, Anh Tran for helpful discussions, and
especially Joan Porti for providing the proof of
Proposition~\ref{abelian comp is reduced}.

\section{Proofs  of Theorems \ref{distinct}, \ref{distinct3} and \ref{degree from trace field}}\label{sec:pf}

\def\Aut{\mathrm{Aut}}
\subsection{Preliminaries}\label{sec.prelim}
Let $\Aut(\c)$ denote the group of all field automorphisms
of the complex field $\c$. Let $\t\in \Aut(\c)$ denote the complex conjugation.

Each element $\phi\in \Aut(\c)$ extends to a unique ring automorphism of the ring $\BC[x_1,\dots, x_n]$ by $\phi(x_i)= x_i$ for $i=1,\dots,n$.
Each element $\phi\in \Aut(\c)$  acts naturally  on
the complex affine space $\c^n$ coordinate-wise  by
$$\phi(a_1,...,a_n):=(\phi(a_1),...,\phi(a_n)).$$
\no{Each element $\phi\in \Aut(\c)$ also acts naturally as an ring automorphism of $\c[x_1,...,x_n]$:
for $f(x_1,...,x_n)=\sum a_{i_1,...,i_n}x_1^{i_1}\cdots x^{i_n}_n\in \c[x_1,...,x_n]$,
$$\phi(f)=\sum \phi(a_{i_1,...,i_n})x_1^{i_1}\cdots x^{i_n}_n.$$}
As a ring automorphism, $\phi$ maps an ideal of $\c[x_1,...,x_n]$ to an ideal, a primary ideal
to a primary ideal and a prime ideal to a prime ideal. If $I\subset \c[x_1,...,x_n]$
is an ideal defined over $\q$, i.e. $I$ is generated by elements
in $\z[x_1,...,x_n]$, then $\phi(I)=I$.
If $V(I)\subset \c^n$ is the zero locus defined by an ideal $I\subset \c[x_1,...,x_n]$,
then $\phi(V(I))=
V(\phi(I))$. We call $\phi(V(I))$ a Galois conjugate of $V(I)$.
 As a map from $\c^n$ to itself,
 $\phi$ maps a variety to a variety, an irreducible
variety to an irreducible variety preserving its dimension.
Furthermore if $V$ is a variety  defined over $\mathbb Q$,
then for any $\BC$-component $V_1$ of $V$, the $\Aut(\c)$-orbit of $V_1$
is the  $\q$-component of $V$  containing $V_1$.
In particular if $V$ is $\q$-irreducible,
then its $\BC$-components    are the $\Aut(\c)$-orbit of one of them
and thus all have the same dimension. (cf. \cite[Section 5]{BZ2}).

%, the only element of $\Aut(\c)$ which is continuous in the $\c$-topology.

A variety is {\em 1-equidimensional} if every its $\BC$-component has dimension 1. A rational map $f: V_1\to V_2$ between two 1-equidimensional varieties is said to have degree $d$ of there is an open dense subset $V_2' \subset V_2$ such that $f^{-1}(V_2')$ is dense in $V_1$ and $f^{-1}(x)$ has exactly $d$ elements for each $x\in V_2'$. When $V_1, V_2$ are $\c$-irreducible, this definition is the same as the well-known definition of a degree $d$ map in algebraic geometry \cite{S}. It is known that a rational map between two $\c$-irreducible varieties is birational if and only if it has degree 1.
\def\oXM{\overline X(M)}
\subsection{Proof of Theorem \ref{distinct}}\label{G-action}
By Mostow-Prasad rigidity, $\oXM$ has  two discrete faithful characters, which are related by the $\t$-action. Let
 $\scz$ be one of the two  discrete faithful characters, then $\t(\scz)$ is the other
one.
By \cite[Corollary 3.28]{P} (which is also valid in $PSL_2(\c)$-setting), each of
$\scz$ and $\t(\scz)$ is a smooth point of $\overline X(M)$. In particular
each of them is contained in a unique $\BC$-component  of $\overline X(M)$, which has dimension 1 (a curve) by a result of
Thurston (see \cite[Proposition 1.1.1]{CGLS}).

We may assume that  $\overline X_1$ is the $\BC$-component
of $\overline X(M)$
which contains $\scz$.
It follows obviously that $\overline X^{\rg}(M)$ (whose definition is given
in Section \ref{intro}) is the $\Aut(\c)$-orbit of
$\overline X_1$, and thus is irreducible over $\q$.
Furthermore each $\BC$-component of $\overline X^{\rg}(M)$    is a
curve.
Hence we have proved part (1) of Theorem \ref{distinct}

By \cite[Theorem 3.1]{D}, $\iota^*: \overline X_1\ra \overline Y_1$ is a birational isomorphism.
\no{ Note that a rational map between two
irreducible algebraic curves is a birational isomorphism iff
it is a degree one map, i.e. a one-to-one map away from
finitely many points of the curves.}
For each  $j=2,...,l$, there is $\phi_j\in \Aut(\c)$ such that $\overline X_j=\phi_j(\overline X_1)$.
 Since $\iota^*$ is defined over $\q$, we have the following commutative
diagram of maps:
$$
\begin{diagram}[height=2em,w=3em]
   \overline X_1&\rTo^{\iota^*}&\overline Y_1\\
\dTo^{\phi_j}&&\dTo_{\phi_j}\\
\overline X_j&\rTo^{\iota^*}&\overline Y_j.
\end{diagram}$$

As $\phi_j$ is a bijection and $\iota^*:\overline X_1\ra \overline Y_1$
is a degree one map,
$\iota^*: \overline X_j\ra \overline Y_j$ is a degree one map and thus
is a birational isomorphism for each $j$.
This proves part (2) of Theorem \ref{distinct}.

Now we proceed to prove part (3) of Theorem \ref{distinct}. By our assumption, both $\scz$ and $\tau(\scz)$ are contained in $\overline X_1$.

\begin{prop}\label{different}For each $j=2,...,l$,   $\overline Y_j$ and $\overline Y_1$ are two distinct curves.
\end{prop}

\pf Suppose otherwise that $\overline Y_1=\overline Y_j$ for some $j\geq 2$. We will get a contradiction from this
assumption.
The argument goes by applying the theory of volumes of representations.

We first recall some of the results from \cite{D} concerning
volumes of representations. For any (connected) closed $3$-manifold $W$ and
any representation $\overline\r\in \overline X(W)$,  the volume $v(\overline\r)$ of $\overline \r$ is defined, and if in addition $\overline \r$ is irreducible,
the volume function $v$ descends down to defined on $\scrp$ so that
 $v(\scrp)=v(\overline \r)$.
 What's important in this theory is
the Gromov-Thurston-Goldman Volume Rigidity
 (proved in \cite{D} as Theorem 6.1), which states that when
 $W$ is a closed hyperbolic $3$-manifold and $\sc\in \overline X(W)$ is an irreducible character,
 then $|v(\sc)|=vol(W)$ iff  $\sc$ is a
 discrete faithful character.
For a hyperbolic knot manifold $M$,
 the volume function $v$
is well defined, in our current notation,  for  each
$PSL_2(\c)$-representation $\overline \r$ of $\pi_1(M)$ whose character $\scrp$ lies in
$\overline X_1$ (\cite[Lemma 2.5.2]{D}). Similarly   if $\scrp\in \overline X_1$ is an irreducible character, then  $v(\scrp)=v(\overline \r)$. So $v$ is defined at
all but finitely many points
of $\overline X_1$. Furthermore  if   $\overline Y_1^\nu$ is a normalization  of $\overline Y_1$
 and $f_1:  \overline Y_1\ra \overline Y_1^\nu$ a birational isomorphism,
  then the volume function $v$ factors through  $\overline Y_1^\nu$
    in the sense that there is a function $v_1:\overline Y_1^\nu\ra \mathbb R$
     such that if $\scrp \in \overline X_1$ is an irreducible character and
     if $f_1$ is defined at $\iota^*(\scrp)$, then
 $$v(\scrp )=v_1(f_1(\iota^*(\scrp ))).$$
That is, we have the following commutative diagram of maps (at points where all maps are defined):
$$
\begin{diagram}[height=2.5em,w=3em]
           & & && \overline Y_1^\nu     \\
 &&&\ldTo(4,2)^{v_1}&\uTo^{f_1}\\
  \mathbb R&\lTo^{\;\;\;\;\;v}&\overline X_1 &\rTo^{\iota^*} & \overline Y_1.
\end{diagram}
$$
This is \cite[Theorem 2.6]{D}.
Moreover if $\scrp\in \overline X_1$ is an irreducible character
such that $\bar\r$ factors through the fundamental
group of a Dehn filling $M(\g)$ of $M$ for some slope $\g$ on $\p M$,
then the volume of $\overline \r$ with respect to $M $
is equal to the volume of $\overline \r$ with respect to the closed manifold
$M (\g)$ (\cite[Lemma 2.5.4]{D}). We note that in the above cited results of \cite{D}
the volume $v(\bar\r)$ is the absolute value of the integral over $M$ (or $M(\g)$)
of certain $3$-form associated to $\bar\r$ but all these results remain valid
  when $v(\bar\r)$ is defined to be the mentioned integral without taking the absolute value. It is this latter version of volume function that we are using here and
 subsequently.

In \cite{F}, \cite[Lemma 2.5.2]{D} is generalized
and it is showed there that the volume function $v$ is well defined at every $PSL_2(\c)$-representation
of a finite volume   hyperbolic $3$-manifold, and
also in \cite{F} the volume rigidity is extended
to all hyperbolic link manifolds,  which states that the volume of a representation
of a hyperbolic link manifold attains its maximal value in absolute value precisely when the representation is
discrete faithful and the maximal value in absolute value is the volume of
the hyperbolic link manifold.
That means that  in our current case, the volume function $v$ is defined
at any irreducible character of $\overline X(M )$ without
the restriction that the character lies in a $\BC$-component of $\overline X(M )$
which contains a discrete faithful character.
We should also note that the definition of the volume of a representation
defined in \cite{F} is consistent with that defined in \cite{D}
in case of a knot manifold. More specifically for a
knot manifold $M$ and a representation $\bar \r\in \overline{R}(M)$,
the volume $vol(\bar\r)$ of $\bar\r$ is defined
through a so called pseudodeveloping map for $\bar \r$ which is
 defined in \cite{D} and
   the independence of $vol(\bar\r)$ from the choice of  the
pseudodeveloping map   is  proved in \cite{D} when
$\chi_\r$ is contained in a $\c$-component of $\oX(M)$
which contains a discrete faithful character and
proved in \cite{F} without any restriction.
One can then  check that the results of \cite{D} which we recalled in
the preceding paragraph can be extended to the following theorem.

\begin{thm}\label{volume theorem}
 (1)   If   $\overline Y_j^\nu$ is a normalization  of $\overline Y_j$
 and $f_j:  \overline Y_j\ra \overline Y_j^\nu$ is a birational isomorphism,
  then there is a function $v_j:\overline Y_j^\nu\ra \mathbb R$
    which makes the following  diagram of maps commutes
    (at points where all the maps are defined):
$$
\begin{diagram}[height=2.5em,w=3em]
           & & && \overline Y_j^\nu     \\
 &&&\ldTo(4,2)^{v_j}&\uTo^{f_j}\\
  \mathbb R&\lTo^{\;\;\;\;\;v}&\overline X_j &\rTo^{\iota^*} & \overline Y_j
\end{diagram}$$

(2) If $\sc\in \overline X(M)$ is an irreducible character which factors through
a Dehn filling $M(\g)$, i.e. $\sc\in \overline X(M(\g))$, then the volume of
$\sc$ with respect to $M$ is the same volume with respect  to $M(\g)$.
If in addition that $M(\g)$ is hyperbolic, then
$|v(\sc)|=vol(M(\g))$ iff $\sc$ is a discrete faithful character of
$M(\g)$.

(3) For any irreducible character  $\sc\in\overline X(M)$,
 $|v(\sc)|\leq vol(M)$, and the equality attains exactly at
  the two discrete faithful characters of $M$.
 \end{thm}

\begin{rem}\label{rem.lift}
{\rm For   the proof of part (1) of the theorem, following that of \cite[Theorem 2.6]{D}, one needs the property that the curve $\overline X_j\subset \overline X(M)$ lifts to a curve in $X(M)$.
But that follows from the fact that $\overline X_1$ lifts (by Thurston)
to a curve  in $X(M)$, say $X_1$,  and then $\phi_j(X_1)$ is a lift of $\overline X_j=\phi_j(\overline X_1)$.}
\end{rem}

We now continue  to prove Proposition \ref{different}.
Take a sequence of distinct slopes $\{\g_k\}$  in  $\p M $,
 and let
$M (\g_k)$ be the closed $3$-manifold obtained by  Dehn filling  $M $ with the slope $\g_k$.
By Thurston's hyperbolic Dehn filling Theorem,
we may assume that $M (\g_k)$ is hyperbolic and that the core circle of the filling solid torus is a geodesic,   for each  $k$.
Note that $\overline X(M (\g_k))\subset \overline X(M )$ for each $k$.
Also for each $k$,  $\overline X(M (\g_k))$
contains precisely two discrete faithful characters, which we denote by
$\sck$ and $\t(\sck)$.
Again by Thurston's hyperbolic Dehn filling theorem,
we may assume that $\sck\ra \scz$ in $\overline X(M )$ with respect to the classical topology
of $\overline X(M )$, up to replacing $\sck$ by $\t(\sck)$ for some $k$'s.
It follows that $\sck$ is contained in $\overline X_1$ for all sufficiently large $k$.

Note that for any irreducible character $\sc\in \overline X(M)$,
$v(\sc)=-v(\t(\sc))$ (see, e.g. \cite[Proposition 4.16]{F}).
Without loss of generality we may assume that $v(\scz)=vol(M)>0$
and so $v(\t(\scz))=-vol(M)$.
It follows that $v(\sck)=vol(M(\g_k))>0$ and $v(\t(\sck))=
-vol(M(\g_k))<0$, at least  for all sufficiently large $k$.

As $\scz$ and $\t(\scz)$ are smooth points of $\oX(M)$,
it follows that $\t(\overline X_1)=\overline X_1$
and that $\t(\sck)$ approaches $\t(\scz)$ as $k\ra \infty$ since $\t$ is
a continuous map.
Therefore $\overline X_1$ is the only $\BC$-component of $\overline X(M)$ which contains
$\sck$ and $\t(\sck)$ for all sufficiently large $k$.

 Since  $\overline Y_1=\overline Y_j$ and the map $\iota^*:\overline X_j\ra \overline Y_1$
  is an almost onto map, there are two sequences of points  $\{\sckp\}$ and $\{\sckpp\}$
  in $\overline X_j$
  such that   $\iota^*(\sckp)=\iota^*(\sck)$ and $\iota^*(\sckpp)=\iota^*(\tau(\sck))$ for almost all  $k$.
  We may also assume that $\sckp$ and $\sckpp$ are irreducible for almost all $k$
 since there are at most finitely many reducible characters in
 $\overline X_j$.

Let $\overline Y_1^\nu$ be a normalization of $\overline Y_1$, and let $f_1:
\overline Y_1\ra \overline Y_1^\nu$
be a birational isomorphism.   As $f_1$ is defined on $\overline Y_1$ except for possibly finitely many points,
we may assume that $f_1$ is well defined at $\iota^*(\sckp)=\iota^*(\sck)$ and $\iota^*(\sckpp)=\iota^*(\tau(\sck))$ for  all large $k$.

Let $v_1$ and $v_j$ be the functions on $\overline Y_1^\nu$ provided by
part (1) of Theorem \ref{volume theorem} with respect to the map $f_1:\overline Y_1\ra \overline Y_1^\nu$.
Note that away from a finitely many points in $\overline Y_1$,  $v_1\circ f_1$ and $v_j\circ f_1$ are smooth functions and have the same differential, up to sign.
(cf. the proof of \cite[Theorem 2.6]{D} for this assertion. Briefly, on 
$(\c^\times)^2$ there is a real
valued $1$-form: $$\o=-\frac{1}{2}(\log |\sl|d\; arg(\sm)-\log|\sm|d\;arg(\sl))$$
which is defined in \cite{CCGLS}. This $1$-from is invariant under 
the involutions $\s$ and $\e_1^*$ on $(\c^\times)^2$ defined in Subsection \ref{sec.Apoly} and thus descends to
a $1$-form $\o'$ on $\oX(\p M)$. For each $j$, $d(v_j\circ f_j)$ is equal to the restriction of $\o'$  over an open dense subset of $\overline Y_j$, up to sign.)
 It follows that $$v_j\circ f_1=\d (v_1\circ f_1)+c$$ for some $\d\in\{1,-1\}$ and  some constant $c$,  in the complement of finitely many points in $\overline Y_1$.
Let $U$ denote this complement.
Then we may assume that $\iota^*(\sckp)=\iota^*(\sck)$ and $\iota^*(\sckpp)=\iota^*(\tau(\sck))$ are contained in $U$ for all large $k$.

Hence  $$v(\sckp)=v_j(f_1(\iota^*(\sckp)))=\d v_1(f_1(\iota^*(\sck)))+c=\d v(\sck)+c$$
and  $$v(\sckpp)=v_j(f_1(\iota^*(\sckpp)))=\d v_1(f_1(\iota^*(\tau(\sck))))+c=\d v(\tau(\sck))+c.$$
So $$v(\sckp)-v(\sckpp)=
\d [v(\sck)- v(\tau(\sck))]=\d 2v(\sck)=\d 2vol(M(\g_k))$$
and thus
$$|v(\sckp)|+|v(\sckpp)|\geq  2vol(M(\g_k))$$
for sufficiently large $k$.
Because $\iota^*(\sckp)=\iota^*(\sck)$ and $\iota^*(\sckpp)=\iota^*(\tau(\sck))$,  $\sckp$ and
$\sckpp$ are both characters of $\overline X(M (\g_k))$.
To see this in detail, let $\overline \r_k, \overline \r_k' \in \overline R(M)$ be representations
with $\sck$ and $\sckp$ as characters respectively.
Note that $\overline \r_k$ is a discrete faithful representation of $\pi_1(M(\g_k))$
and so  $\overline \r_k(\g_k)=1$. Let $\eta_k$ be a simple essential loop in $\p M$ such that
$\{\g_k,\eta_k\}$ form a basis of $\pi_1(\p M)$. Then $\eta_k$ is isotopic in $M(\g_k)$ to the
core circle of the filling solid torus in forming $M(\g_k)$ from $M$.
As we have assumed that the core circle is a geodesic in the hyperbolic $3$-manifold
$M(\g_k)$, $\overline \r_k(\eta_k)$ is a hyperbolic element of $PSL_2(\c)$. In particular its trace
square is not equal to $4$.
Now since $\sckp(\g_k)=\sck(\g_k)$ and $\sckp(\eta_k)=\sck(\eta_k)$, $\overline \r_k'(\g_k)$
is a parabolic element or the identity element and $\overline \r_k'(\eta_k)$ is a hyperbolic element of $PSL_2(\c)$.
But these two elements commute, $\overline \r_k'(\g_k)$ has to be the identity element.
Hence $\sckp\in \overline X(M(\g_k))$. Similarly one can show that
$\sckpp\in \overline X(M(\g_k))$.
But neither $\sckp$ nor $\sckpp$ is a discrete faithful character of $\overline X(M (\g_k))$
by our construction, we get a contradiction with
the volume rigidity theorem for closed hyperbolic $3$-manifolds.
\qed

To finish the proof of  Theorem \ref{distinct} (3), we jus need to show that $\overline Y_j$,
$j\geq 2$ are mutually distinct.
Suppose  that
 $\overline Y_{j_1}=\overline Y_{j_2}$ for some $j_1, j_2\geq 2$.
 There is $\phi\in \Aut(\c)$ such that
 $\phi(\overline X_{j_1})=\overline X_{1}$.
 As $\phi$ commutes with $\iota^*$,
 $\phi(\overline Y_{j_1})=\overline Y_{1}$.
 We also have $\phi(\overline Y_{j_2})=
 \phi(\overline Y_{j_1})=\overline Y_1$.
  So by Proposition \ref{different},
 $\phi(\overline X_{j_2})=\overline X_1$ as well.
Hence $\overline X_{j_1}=\overline X_{j_2}$, i.e.  $j_1=j_2$.

\def\oY{\overline Y}
\subsection{A refinement of Theorem \ref{distinct}} Let $M$ be a hyperbolic knot manifold.
 Suppose $\oX_1,\dots,\oX_k$ are all $\c$-components of
$\oX(M)$ and  $\oY_j$ is the Zariski closure of $\iota^*(\oX_j)$ in $\oX(\partial M)$ for $i=1,\dots,k$. It is known that $\oY_j$ has dimension either 1 or 0.

In proving $\oY_j \neq \oY_1$ in the previous subsection, the fact that $\oX_j$ is in the $\Aut(\c)$-orbit of $\oX_1$ is used only to show that $\oX_j$ lifts to $X(M)$ (see Remark \ref{rem.lift}).
When $M$ is a hyperbolic knot manifold which is the exterior of a knot
in a homology $3$-sphere, every $PSL_2(\c)$-representation of $\pi_1(M)$ lifts to a $SL_2$-representation.
Hence, the proof of part (3) of Theorem \ref{distinct} also proves the following.
\begin{thm} \label{distinct2} Let $M$ be a hyperbolic knot manifold which is the exterior of a knot
in a homology $3$-sphere. Let $\overline X_1, \dots,\overline X_k$ be the $\c$-components of the $PSL_2$-character varieties $\overline X(M)$, and $\overline Y_j$ be the Zariski closure of $\iota^*(X_j)$, $j=1,\dots, k$.
Suppose the two discrete faithful characters of
$\overline X(M)$ are contained in  $\overline X_1$.
Then $\overline Y_j \neq \overline Y_1$ for all $j \ge 2$.
\end{thm}

\subsection{Proof of  Theorem \ref{distinct3}}\label{e-action}
Let $(\m, \l)$ be the standard meridian-longitude
basis for $\pi_1(\p M )\subset \pi_1(M )$. Let $\BZ_2=\{1,-1\}
% \subset GL_1(\c)= \c^\times :=\c \setminus \{0\}
 $. Since $H^1(M,\BZ_2)=\BZ_2$, there is a unique non-trivial group homomorphism $\ve: \pi_1(M)\to \BZ_2$.
 %, which gives a complex 1-dimensional representation of $\pi_1$.
 One has $\ve(\mu)=-1$ and $\ve(\lambda)=1$.
The homomorphism $\e$ induces an involution $\e^*$ on $R(M )$ and on $X(M )$, defined by $\ve^*(\rho)(\g)= \ve(\g) \, \rho(\g)$ for $\rho\in R(M)$ and $\e^*(\scr)=\sce$
 for $\scr\in X(M )$.
\no{ The restriction of $\ve$ onto $\pi_1(\partial M) \subset \pi_1(M)$ similarly gives rise to an involution
 $\ve^*$ on $X(\partial M)$.}

Obviously $\e^*$ is a bijective
regular involution  on $X(M )$\no{ and on $X(\partial M)$}, and it is defined
over $\q$.
The quotient space of $X(M )$ \no{(respectively  $X(\partial M)$)} by this involution gives rise a regular map $\Phi^*$ from $X(M)$ \no{(respectively  $X(\partial M)$)} into $\overline X(M)$ \no{(respectively $\overline X(\partial M)$)}. Let $\Phi:SL_2(\c)\ra PSL_2(\c)$ be the canonical quotient
homomorphism. Then $\Phi^*$ is exactly the  map induced by $\Phi$.

On the other hand, since $H_1(M; \z_2)=\z_2$, every $PSL_2(\c)$-representation
$\overline \r$ of $M$ lifts to a $SL_2(\c)$ representation $\r$ of $M$ in the sense
that $\overline\r=\Phi\circ \r$ (cf. e.g. \cite[Page 756]{BZ}).
Hence $\Phi^*$ is an onto map on $X(M)$. \no{On $X(\partial M)$, $\Phi^*$ is also onto.}

Similarly if $\e_1:\pi_1(\p M)\ra \z_2=\{1,-1\}$ is the homomorphism defined by
$\e_1(\m)=-1$ and $\e_1(\l)=1$, it induces an involution $\e_1^*$ on $X(\p M)$.
Let $\Phi_1^*$ be the corresponding quotient map from $X(\p M)$
into $\overline X(\p M)$. Then $\Phi_1^*$ is also a regular and surjective map.
We have the following commutative diagrams of regular maps:
\be
\label{dia.1}
\begin{diagram}[height=2em,w=3em]
 X(M)&\rTo^{\iota^*}& X(\p M)\\
\dTo^{\Phi^*} &&\dTo_{\Phi^*_1}\\
  \overline X(M)&\rTo^{\iota^*}&\overline X(\p M)
\end{diagram}
\ee
and the upper $\iota^*$ satisfies the identity 
\begin{equation}\label{epsilon equivariance}
\e_1^*\circ\iota^*=\iota^*\circ \e^*.
\end{equation}
 In particular we have
\be\label{dia.2}
\begin{diagram}[height=2em,w=3em]
 X^{\rg}(M)&\rTo^{\iota^*}& Y\\
\dTo^{\Phi^*} &&\dTo_{\Phi^*_1}\\
  \overline X^{\rg}(M)&\rTo^{\iota^*}&\overline Y.
\end{diagram}
\ee
where $\overline Y$ is
the Zariski closure of $\iota^*(\overline X^{\rg}(M))$
and $ Y$ is
the Zariski closure of $\iota^*(X^{\rg}(M))$. All the varieties in Diagram \eqref{dia.2} are 1-equidimensional.  Each of $\Phi^*$ and $\Phi_1^*$ is a degree $2$ map, while the lower $\iota^*$ has degree $1$ by Theorem \ref{distinct}. It follows that  there are an open dense subset $(\oX^\rg)'$ of $\oX^\rg(M)$ and an open dense subset $\oY'$ of $\oY$ such that the restriction
$\iota^*:(X^{\rg})'\to\overline Y'$ is a bijection and
for every $x \in (\oX^\rg)'$ and $y\in \oY'$, $(\Phi^*)^{-1}(x)$ has exactly two distinct elements and so does $(\Phi_1^*)^{-1}(y)$.
Besides, $(X^\rg)':= (\Phi^*)^{-1}((\oX^\rg)')$ is open dense in $X^{\rg}(M)$ and $Y'= (\Phi^*_1)^{-1}(\oY')$ is open dense in $Y$. Suppose $x\in (\oX^\rg)'$, $y= \iota^*(x)$, $\{x_1, x_2\}= (\Phi^*)^{-1}(x)$, and
$\{y_1, y_2\}= (\Phi^*_1)^{-1}(y)$. The commutativity of the above diagram means $\iota^*(x_1)$ is one of $y_1, y_2$, say $\iota^*(x_1)=y_1$. Then the identity (\ref{epsilon equivariance}) implies $\iota^*(x_2)=y_2$. This shows that $\iota^*$ is a bijection from the open dense subset $(X^\rg)'$ of $X^\rg(M)$ onto the open dense subset $Y'$ of $Y$. Hence,
$\iota^*: X^{\rg}(M)\ra Y$ is a degree one map. Now Theorem \ref{distinct3} follows from Theorem \ref{distinct}.
\begin{remark}\label{r.equal}
 Let $M$ be a hyperbolic knot manifold which is the exterior of a knot
in a homology $3$-sphere. We saw in the above proof that $\Phi^*$ is surjective on $X(M)$.
 Since $(\Phi^*)^{-1}(\oX^\irr(M))= X^\irr(M)$ and $(\Phi^*)^{-1}(\oX^\rg(M))= X^\rg(M)$, we conclude that
 $X^\irr(M)= X^\rg(M)$ if and only in $\oX^\irr(M)= \oX^\rg(M)$.
\end{remark}
\subsection{$A$-polynomial and its symmetry} \label{sec.Apoly}
We  briefly recall the definition of the $A$-polynomial
for a knot $K$ in a homology $3$-sphere $W$, as defined in \cite{CCGLS}. Let $M$ be the exterior of $K$
and let $\{\m, \l\}$ be the standard meridian-longitude basis for $\pi_1(\p M)$.

\def\BCx{\BC^\times}
\def\BCxt{(\BC^\times)^2}
\def\GG{\langle \sigma, \ve_1^*\rangle}
\def\pr{\mathrm{pr}}
Let $\BCx = \BC\setminus \{0\}$ and  $\sigma:(\BCx)^2 \to (\BCx)^2$ be the involution defined by $\sigma(\sm,\sl)= (\sm^{-1}, \sl^{-1})$. We can identify $X(\partial M)$ with $\BCxt/\sigma$ as follows.
For $(\sm,\sl)\in \BCxt$ let $\chi_{(\sm,\sl)}\in X(\partial M)$ be the character of the representation
$$ \rho: \pi_1(\partial M)\to SL_2(\BC), \quad \displaystyle\r(\m)=\left(\begin{array}{cc}\sm&0\\0&\sm^{-1}
\end{array}\right), \displaystyle\r(\l)=\left(\begin{array}{cc}\sl&0\\0&\sl^{-1}
\end{array}\right).$$
\def\prs{\mathrm{pr}_\sigma}
Then the map $(\sm,\sl)\to \chi_{(\sm,\sl)}$ descends to an isomorphism from $\BCxt/\sigma$ onto $X(\partial M)$, which we use to identify $\BCxt/\sigma$ with $X(\partial M)$. Let $\prs: \BCxt \to \BCxt/\sigma\equiv X(\partial M)$ be the natural projection.

Let $\ve_1^*:\BCxt\to \BCxt$ be the involution defined by $\ve_1^*(\sm,\sl)=(-\sm,\sl)$. Then $\ve_1^*$ commutes with $\sigma$ and
descends to an involution of $\BCxt/\sigma$, which coincides with the $\ve_1^*$ of Section~\ref{e-action}.
%The involution $\ve^*: X(\partial M) \to X(\partial M)$ now has the form $\ve^*(\sm,\sl)=(-\sm,\sl)$.
 Thus, we can identify $\oX(\partial M)$ with $(\BCxt/\sigma)/\ve_1^*= \BCxt/\langle \sigma, \ve_1^*\rangle$, and $\Phi_1^*$ with the natural projection $\BCxt/\sigma \to \BCxt/\GG$. Here $\GG\cong \BZ_2 \times \BZ_2$ is the group generated by $\sigma$ and $ \ve_1^*$. Let $\pr:\BCxt\to \BCxt/\GG$ be the natural projection.

\def\cR{\mathcal R}
The involution $\sigma$ naturally induces an algebra involution, also denoted by $\sigma$, acting on the algebra $\BC[\sm^{\pm1}, \sl^{\pm1}]$. That is,
$ \sigma (P)(\sm,\sl)= P(\sm^{-1},\sl^{-1})$ for $P\in \BC[\sm^{\pm}, \sl^{\pm1}]$.
A polynomial $P\in \BC[\sm, \sl]$ is said to be {\em balanced} if $\sigma(P)= \delta \sm^a \sl^b P$ for certain $\delta \in \{-1,1\}$ and $a,b \in \BZ$.
For any subring $\cR\subset \BC[\sm,\sl]$, we say that $P\in \BC[\sm,\sl]$ is {\em balanced-irreducible in} $\mathcal R$ if $P\in \cR$ and $P$ is balanced but is not the product of two  non-constant  balanced polynomials in $\cR$.

\def\oX{\overline X}
\def\oXM{\overline X(M)}
\def\oY{{\overline Y}}
\def\oZ{{\overline Z}}

\def\Aut{\mathrm{Aut}}

Suppose $Z\subset X(\partial M)$ is 1-equidimensional variety. The Zariski closure $\tilde Z$ of $\prs^{-1}(Z)$ in $\BC^2$ is a 1-equidimensional variety. The ideal of all polynomials in $ \BC[\sm, \sl]$ vanishing on $\tilde Z$ is principal, and is generated by a polynomial $P_Z\in \BC[\sm, \sl]$, defined up to a non-zero constant factor. The $\sigma$-invariance of $\prs^{-1}(Z)$ implies that $P_Z$ is balanced. If $Z$ is $\BC$-irreducible, then $P_Z$ is balanced-irreducible in $\BC[\sm,\sl]$.
 If
 $Z$ is defined over $\q$, then one can choose $P_Z\in \BZ[\sm, \sl]$ and it is defined up to sign.   If $Z$ is $\BQ$-irreducible, then $P_Z$ is balanced-irreducible in $\BZ[\sm,\sl]$.

Similarly, suppose $\oZ\subset \oX(\partial M)$ is 1-equidimensional variety, one defines $P_\oZ\in \BC[\sm, \sl]$ as the generator of the ideal of all polynomials in $ \BC[\sm, \sl]$ vanishing on $\pr^{-1}(\oZ)$. The $\GG$-invariance of $\pr^{-1}(\oZ)$ implies that  $P_\oZ$ is balanced and  belongs to $ \BC[\sm^2, \sl]$. If $\oZ$ is $\BC$-irreducible, then $P_\oZ$ is balanced-irreducible in $\BC[\sm^2,\sl]$.
 If
 $\oZ$ is defined over $\q$, then one can choose $P_\oZ\in \BZ[\sm^2, \sl]$ and it is defined up to sign.   If $\oZ$ is $\BQ$-irreducible, then $P_\oZ$ is balanced-irreducible in $\BZ[\sm^2,\sl]$.

Now let $\oZ$ be the union of all one-dimensional $\BC$-components of the Zariksi closure of $\iota^*(\oX(M))$ in $\oX(\partial M)= \BCxt/\GG$. It is known that $\oZ$ is defined over $\BQ$. The polynomial
$P_\oZ \in \BZ[\sm^2,\sl]$ is the $A$-polynomial $\displaystyle A_{\sk, \sw}(\sm,\sl)$.
If $Z$ is the union of all one-dimensional $\c$-components of the Zariaski closure of $\iota^*(X(M))$ in $X(\partial M)$.
Then $Z=(\Phi^*)^{-1}(\oZ)$. Thus $P_Z=P_\oZ$ is the $A$-polynomial.

\begin{rem}To define the $A$-polynomial $A_{\sk,\sw}(\sm,\sl)$, one just needs to
consider in $SL_2(\c)$-setting, i.e. in terms of  $P_Z$, as how it's done in \cite{CCGLS}.
For our purpose (e.g. for a convenience in proving   Proposition \ref{bal-irred}) we also present  the same $A$-polynomial from $PSL_2(\c)$ point of view, i.e. in terms of $P_\oZ$.
\end{rem}

\subsection{Proof of Theorem \ref{degree from trace field}}

Since $M$ is the exterior of a knot $K$ in a homology $3$-sphere,
its trace field is equal to its invariant trace field.
Let $\sz$ be a discrete faithful character of $X(M)$.
It is proved in \cite{SZ} that
the $\Aut(\c)$-orbit of $\sz$ has $d$ distinct elements,
which we denote by $\szj$, $i=0,1,..., d-1$,
and $\e^*(\szj)$, $i=0,1,....,d-1$, is another set of $d$ distinct elements,
which is disjoint from the former set.
These $2d$ characters are obviously contained in
$X^{\rg}(M)$. Furthermore they are irreducible faithful characters whose values
on elements of $\pi_1(\p M)$ are $2$ or $-2$.

For $\g\in \pi_1(M)$, let $f_\g$ be the regular function on $X(M)$ defined
by $f_\g(\scr)=[trace(\r(\g))]^2-4$. Then by the discussion above,
for each peripheral element $\g\in \pi_1(\p M)$, $f_\g$ has at least
$2d$
zero points: $\szj, \e^*(\szj)$, $i=0,...,d-1$.

Now let $X_1, ..., X_l$ be the $\BC$-components of $X^{\rg}(M)$.
By \cite[Section 5]{BZ2}, $f_\g$ is non-constant on each $X_j$ for
every nontrivial element $\g\in \pi_1(\p M)$ and
 the degree of $f_\g$ on $X_j$ remains  the same for $j=1,\dots,l$.
It is shown in \cite{SZ} that
\be
\label{eq.11}
\displaystyle\sum_{j=1}^l\left. degree(f_\g)\right|_{X_j}\geq 2d.
\ee
 Perhaps  we need to note that
  $\displaystyle\left. degree(f_\gamma)\right|_{X_j}$
  is equal to  the Culler-Shalen norm
 of $\g\in \pi_1(\p M)$  defined by the curve $X_j$, and
  the inequality \eqref{eq.11} is  given in \cite{SZ} in terms
  the Culler-Shalen norm.

Let $P_j=P_{Y_j}\in \BC[\sm,\sl]$ (see the definition of $P_Z$ in Subsection \ref{sec.Apoly}), where $Y_j$ is the Zariski
closure of $\iota^*(X_j)$.
    Theorem \ref{distinct3} part (2) and \cite[Proposition 6.6]{BZ2}
together imply that the $\sm$-degree of $P_j(\sm,\sl)$
is equal to $\displaystyle\left. \frac 12 degree(f_\l)\right|_{X_j}$ and
the $\sl$-degree of $P_j(\sm,\sl)$
is equal to $\displaystyle\left. \frac 12 degree(f_\m)\right|_{X_j}$.
We note at this point that although the definition of $A$-polynomial
defined in \cite{BZ2} is a bit different from
that given in \cite{CCGLS}, when
the degree of the map $\iota^*|_{X_j}$ is one
the factor $P_j(\sm,\sl)$ contributed by $X_j$
 is the same polynomial either as defined in \cite{CCGLS} or as
 defined in \cite{BZ2}.
 Since we do have that $\iota^*|_{X_j}$ is a degree one map,
  \cite[Proposition 6.6]{BZ2} applies.

Moreover it follows from Theorem \ref{distinct3} part (3) that all
factors $P_j(\sm,\sl)$, $j=1,...,l$, are mutually distinct.
Hence the $\sm$-degree and the $\sl$-degree of $A_{\sk,\sw}(\sm,\sl)$
are  larger than or equal to  $\displaystyle\sum_{j=1}^l\left. \frac 12 degree(f_\l)\right|_{X_j}\;\;\;\mbox{and}\;\;\;
\displaystyle\sum_{j=1}^l\left. \frac 12 degree(f_\m)\right|_{X_j}$ respectively, which are bigger than or equal to $d$ by \eqref{eq.11}.
This completes the proof of Theorem \ref{degree from trace field}.

\subsection{$A$-polynomial and balanced-irreduciblity} The following will be used in the proof of Theorem \ref{AJ}.

\begin{prop}\label{bal-irred}
Suppose that $M$ is a hyperbolic knot manifold which is the exterior of a knot $K$
in a homology $3$-sphere $W$. Assume that  the two discrete faithful characters are in the same $\BC$-component of the $PSL_2(\BC)$-character varietiy $\oX(M)$ and
 $\oX^\irr(M)=\oX^{\rg}(M)$. Then the non-abelian $A$-polynomial
$\widehat A_{\sk,\sw}(\sm,\sl)$ is   non-constant, does not contain any $\sm$-factor or $\sl$-factor, and is  balanced-irreducible in
$\BZ[\sm^2,\sl]$.
\end{prop}

Here an $\sm$-factor (resp. $\sl$-factor) means a non-constant element of $\BZ[\sm]$ (resp. $\BZ[\sl]$).

\def\tY{\tilde Y}

\pf Let $\oY$ be the Zariski closure of $\iota^*(\oX^\rg)$ in $\oX(\partial M)$.
Since $\oX$ is $\q$-irreducible, it follows  from  Theorem~\ref{distinct}
that $\oY$ is $\q$-irreducible. Therefore $P_\oY$ is
 balanced-irreducible in
$\BZ[\sm^2,\sl]$ (see Section \ref{sec.Apoly}).
 When $\oX^\irr(M)= \oX^\rg(M)$, $P_\oY$ is the whole $\widehat A_{\sk,\sw}(\sm,\sl)$. Hence $\widehat A_{\sk,\sw}(\sm,\sl)$ is balanced-irreducible in
$\BZ[\sm^2,\sl]$.

Let $X_1, \dots,X_l$ be the $\BC$-components of $X^\rg(M)$.
As pointed out in the proof of Theorem \ref{degree from trace field},  for $j=1,\dots,l$, both of the $\sm$-degree
and the $\sl$-degree of $P_{X_j}(\sm,\sl)$ are positive.
As $P_{X_j}$ is balanced-irreducible, it follows that $P_{X_j}$
cannot contain any  $\sm$-factor or $\sl$-factor.
In particular $P_{X_j}(\sm,\sl)\ne \sl-1$.
Hence $X^{\rg}(M)$ contributes the factor
$\widehat A_{\sk,\sw}(\sm,\sl)=A_{\sk,\sw}(\sm,\sl)/(\sl-1)$
which is a non-constant  and does not
contain any $\sm$-factor or $\sl$-factor.
\qed
\begin{remark}\label{a refinement} The above proof, combined with Theorem \ref{distinct2},  actually yields  the  following stronger statement. Suppose $M$ is a hyperbolic knot manifold
which is the exterior of a knot $K$ in a homology sphere $W$
such that the two $PSL_2(C)$ discrete faithful characters are contained in the same $\c$-component of
$\oX(M)$. Then $\widehat A_{K,W}(\sm,\sl)$ is balanced-irreducible in $\BZ[\sm^2,\sl]$ if and only if $X^\rg$ contains every
$\c$-component of $X(M$) whose image under $ \iota^*: X(M) \to X(\partial M)$ is one dimensional.
\end{remark}

\section{Representation schemes and character schemes}\label{sec.scheme}

\subsection{Reduced and essentially reduced schemes}\label{se.reduced}
Concerning the proof of Theorem \ref{AJ}, we need to consider the scheme counterparts of  the $SL_2(\c)$ representation variety and  character variety of a group $\G$.
Let's  first prepare   some   facts about an
affine scheme $\spec(R)$ for a ring $R$ of the form
$R=\c[x_1,...,x_n]/I$ where $I$ is a proper ideal of $\c[x_1,...,x_n]$.
The ideal $I$ admits an irredundant primary
decomposition, i.e.
$$I =\bigcap_{j=1}^{m}Q_j$$
for some positive integer $m$ such that each
$Q_j$ is a primary ideal and
$\sqrt{Q_i}\ne \sqrt{Q_j}$ for $i\ne j$.
 The radical $P_j =\sqrt{Q_j}$ is a prime ideal. Recall that $Q_j$ is called an isolated
component of $I$ if $P_j$ is minimum in the inclusion relation
among $P_1,...,P_m$, and if $Q_j$ is not isolated,
it is called an embedded component of $I$.
The set $\{P_1,...,P_m\}$ is uniquely determined by $I$, as well as the set of all isolated components $Q_j$ of $I$.
We may assume that $Q_j$, $j=1,...,k$, are the
isolated components of $I$.

Let $V(I)\subset \c^n$ be the  zero locus of an ideal $I\subset \c[x_1,...,x_n]$, which is a  variety. Note that $V(I)=V(\sqrt I)$.
The coordinate ring $\c[V]$  of $V=V(I)$ is given by $\c[V]=\c[x_1,...,x_n]/\sqrt I$
which is also equal to the quotient ring of $R=\c[x_1,...,x_n]/I$ divided by its nilradical $\sqrt{(0)}$, i.e.
$$\c[V]=R/\sqrt{(0)}.$$
The variety $V=V(I)$ can be naturally identified with
the set of closed points of the scheme $\spec(R)$.
The zero loci $V_j=V(Q_j)=V(P_j)$, $j=1,...,k$, are all irreducible $\BC$-components of $V$.
Let $R_j=\c[x_1,...,x_n]/Q_j$, $j=1,...,k$. Then
$\spec(R_j)$, for each $j=1,\dots,k$, is an  irreducible component  of $\spec(R)$, called the {\em component corresponding to $V_j$}.

Recall that a ring is called {\it reduced} if
  it does not contain any non-zero nilpotent
 elements.
 For the ring $R$ above, it is reduced if and only if $I=\sqrt I$.
Similarly the ring $R_j$
is reduced if and only if  $Q_j=\sqrt{Q_j}=P_j$ (or equivalently $R_j$ is an integral domain).
If all $R_j$, $j=1,...,k$, are reduced
(i.e. $Q_j=P_j$ for all isolated components of $I$), we call
the ring $R$ {\it essentially reduced}.
Correspondingly we call an affine scheme $\spec(R)$ reduced if its defining ring $R$ is reduced, and call it {\it essentially reduced} if each
irreducible component $\spec(R_j)$ of
$\spec(R)$ is reduced.

Let $\fm\in \spec(R)$ be a closed point, which
we shall also identify  with a maximal ideal of $R$
as well as with a point in $V$.
Let $T_\fm(\spec(R))$ denote the Zariski tangent space of the scheme
$\spec(R)$ at the point $\fm$ and let $R_\fm$ be the localization of $R$ at the maximal ideal $\fm$. Note that $R_\fm$  is a local ring
and is the stalk of the scheme $\spec(R)$ at the point $\fm$.
If the dimension of $T_\fm(\spec(R))$ is equal to the
(Krull) dimension of the local ring $R_\fm$, then $\fm$ is a {\it smooth point}
of the scheme $\spec(R)$ (called a regular point or a simple point in some textbooks), and the following
conclusions follow: $\fm$ is contained in a unique irreducible
component of $\spec(R)$, say $\spec(R_j)$, $R_\fm=(R_j)_\fm$ is an integral domain, which implies that $R_j$ is an integral domain and thus
 $R_j$ is reduced and $Q_j=P_j$ (see  e.g. \cite{M} \cite{S}).
We summarize this discussion into the following lemma
in a form that is more convenient  for us to apply.

\begin{lemma}\label{dim equailty imples reduceness}Let $R=\c[x_1,..,x_n]/I$ for a proper ideal $I$.
Let $m\in \spec(R)$
be a closed point and let $V_j$ be an irreducible component
of the variety $V=V(I)$ which contains $m$.
Suppose that $dim\; T_\fm(\spec(R))=dim\;V_j$, then
$\fm$ is a smooth point of the scheme $\spec(R)$,
$V_j$ is the unique irreducible component of $V$ which contains $\fm$,  and
the isolated component $Q_j$ of $I$ which defines $V_j$ is a prime ideal.
\end{lemma}

 \pf Let $Q_j$ be the isolated component of $I$
 which defines $V_j$
 and let $R_j=\c[x_1,...,x_n]/Q_j$.
 Then $m\in \spec(R_j)\subset \spec(R)$.
  As we always have  $$dim\;T_\fm(\spec(R))\geq dim\;R_\fm
 =dim\;(R_j)_\fm=dim\;R_j=dim\;R_j/\sqrt{(0)}=dim\;V_j,$$ the assumption
 $dim\;T_\fm(\spec(R))=dim\;V_j$ implies the
 equality $dim\;T_\fm(\spec(R))=dim\;R_\fm$ and thus all the conclusions of the lemma
 follow from the discussion preceding the lemma.
 \qed

\begin{remark}\label{orbit reduce}Recall that every element $\phi\in \Aut(\c)$ induces an action on
$\c[x_1,...,x_n]$. If in the above lemma the ideal $I$ is defined over $\mathbb Q$,  then every element $\phi\in \Aut(\c)$
will keep $I$ invariant, sending isolated components  of $I$
to isolated components, and sending scheme reduced components
$Q_j$ (i.e. $Q_j=\sqrt{Q_j}$ is prime) of $I$ to scheme reduced components.
Hence the $\Aut(\c)$-orbit of a scheme reduced isolated component of $I$
is a set of scheme reduced isolated components of $I$ whose intersection is an ideal defined over $\mathbb Q$.
\end{remark}

\subsection{Character scheme}
Given a finitely presented group $\G$, let $\fA(\G)$ be {\it the universal $SL_2(\c)$  representation ring
of $\G$}, which is a finitely generated $\BC$-algebra (as given by \cite[Proposition 1.2]{LM}, replacing $GL_n$ there by $SL_2$ and $k$ there by $\c$).
\no{Then $\fA(\G)$ is a quotient ring of $\c[x_1,..., x_{4g}]$ for some integer $g>0$ by some ideal $I$ defined over $\q$,
i.e. $$\fA(\G)=\c[x_1,..., x_{4g}]/I,$$ which can be obtained from a presentation of $\G$ (with $g$ generators)
but is independent of the choice of the presentation (as explained in the proof of
\cite[Proposition 1.2]{LM}).}
The $SL_2(\c)$ {\it representation scheme} $\fR(\G)$ of $\G$ is defined to be the scheme $\spec(\fA(\G))$, i.e. $\fR(\G)=\spec(\fA(\G)).$
The set of closed points of $\fR(\G)$ can be identified with  the $SL_2(\c)$  representation variety $R(\G)$ of $\G$.
The coordinate ring $\c[R(\G)]$ of $R(\G)$ can be obtained
as  the quotient  of  $\fA(\G)$ by its nilradical $\sqrt{(0)}$,
i.e. $$\c[R(\G)]=\fA(\G)/\sqrt{(0)}.$$

Induced by  the matrix conjugation, the group $SL_2(\c)$ acts naturally on $\fA(\G)$.
Let $$\fB(\G)=\fA(\G)^{SL_2(\c)}$$ be the subring of invariant
 elements of $\fA(\G)$ under this action, which is finitely-generated as a $\BC$-algebra (by the Hilbert-Nagata theorem).
Then $\fB(\G)$ is called {\it the universal $SL_2(\c)$ character ring of $\G$} and
the scheme  $$\fX(\G):=\spec(\fB(\G))$$ is called {\it the $SL_2(\c)$ character scheme of $\G$}.
\no{ Note that $\fB(\G)$ is isomorphic to a ring of the form
 $\c[x_1,...,x_n]/J$ for some integer $n>0$ and an ideal $J\subset \c[x_1,...,x_n]$.
 }
The set of closed points of  $\fX(\G)$ can be identified with the character variety $X(\G)$ of
$\G$ and the coordinate ring $\c[X(\G)]$ of $X(\G)$ is $\fB(\G)$ divided by
its zero radical, i.e.
$$\c[X(\G)]=\fB(\G)/\sqrt{(0)}.$$

Let $\r\in \fR(\G)=\spec(\fA(\G))$ be a closed point. Then identified
as a point in $R(\G)$, $\r:\G\ra SL_2(\c)$ is a $SL_2(\c)$ representation of $\G$. Similarly the character $\scr\in X(\G)$ of $\r\in R(\G)$ shall  also be considered as a closed point in the character scheme $\fX(\G)=\spec(\fB(\G))$.
  Let $sl_2(\c)$ be the Lie algebra of $SL_2(\c)$, $Ad:SL_2(\c)\ra \Aut(sl_2(\c))$ the
adjoint representation, and $sl_2(\c)_\r$  the $\G$-module $sl_2(\c)$  given by $Ad\circ\r:\G\ra \Aut(sl_2(\c))$.
Then a fundamental observation made in  \cite{W} states that
 the space of group 1-cocycles $Z^1(\G, sl_2(\c)_\r)$ of $\G$
 with coefficients  in $sl_2(\c)_\r$ is naturally isomorphic to the Zariski tangent space $T_\r(\fR(\G))$ of the scheme $\fR(\G)$ at the point $\r$,
and when $\r$ is an irreducible representation and is a smooth point of
$\fR(\G)$, the group $1$-cohomology $H^1(\G, sl_2(\c)_\r)$ is isomorphic to the Zariski tangent space $T_{\scr}(\fX(\G))$
of the scheme $\fX(\G)$ at the point $\scr$ (cf \cite{LM}).

For a compact manifold $W$ we use $\fA(W)$, $\fB(W)$,  $\fR(W)$ and $\fX(W)$ to denote $\fA(\pi_1(W))$, $\fB(\pi_1(W))$, $\fR(\pi_1(W))$ and $\fX(\pi_1(W))$ respectively. When $M$ is a hyperbolic
knot manifold, let $\fX^{\rg}(M)\subset \fX(M)=\spec(\fB(M))$ be the counterpart of $X^{\rg}(M)\subset X(M)$, that is, $\fX^{\rg}(M)$ is the union of the components of $\fX(M)$ corresponding to the $\BC$-components of $X^\rg(M)$.

\begin{prop}\label{canonical comps are reduced}Let $M$ be a hyperbolic
knot manifold. Then  $\fX^{\rg}(M)$ is essentially reduced.
\end{prop}

\pf Let $\scr$ be the character of a discrete faithful
representation of $\pi_1(M)$ and let
 $X_1$ be a $\c$-component of $X(M)$  containing
 $\scr$. By a result of Thurston, $dim\;X_1=1$.
 It is also known that $dim\; H^1(\pi_1(M), sl_2(\c)_\r)=1$
(see \cite{P}). Since $\r$ is an irreducible representation,
the $1$-coboundary $B^1(\pi_1(M), sl_2(\c)_\r)$ is $3$-dimensional
and thus the dimension of $Z^1(\pi_1(M), sl_2(\c)_\r)$ is $4$
which is equal to the dimension of the  $\BC$-component $R_1$ of
$R(M)$ which maps onto $X_1$ under the canonical surjective regular map
$\tr: R(M)\ra X(M)$.
That is  we have $dim\;T_\r(\fR(M))=dim\;Z^1(\pi_1(M), sl_2(\c)_\r))=dim\;R_1$ which means
 by Lemma \ref{dim equailty imples reduceness} that $\r$ is a smooth point of the scheme $\fR(M)$. In turn we have $dim\;T_{\scr}(\fX(M))=dim\;H^1(\pi_1(M), sl_2(\c)_\r)=dim\;X_1$ which means
 by Lemma \ref{dim equailty imples reduceness} again that
$\scr$ is a smooth point of the scheme $\fX(M)$, that $\scr$ is contained in a unique irreducible component $\fX_1$
of $\fX(M)$ which is the scheme counterpart of the component $X_1$
and that $\fX_1$ is reduced.
By Remark \ref{orbit reduce}, the $\Aut(\c)$-orbit of $\fX_1$
consists of reduced components.
As $\fX^{\rg}(M)$ consists of such orbits, each of its components
 is reduced.
\qed

If $M$ is the exterior of a knot  in $S^3$, its set of abelian representations form a unique component $R_0$ of $R(M)$ and $dim \;R_0=3$.
 In fact $R_0$ is isomorphic, as a variety,  to $SL_2(\c)$.
 The image $X_0$ of $R_0$ in $X(M)$ under the quotient map $\tr$
 is a component of $X(M)$ and $dim\;X_0=1$.
 The proof of the following proposition is due to Joan Porti.

\begin{prop}\label{abelian comp is reduced}
Let  $M$ be the exterior of a knot  in $S^3$
and let $\fX_0$ be the unique irreducible component of  $\fX(M)$ corresponding to $X_0$.
Then $\fX_0$   is reduced.
\end{prop}

\pf Note that the meridian element $\m$ generates
the first homology of $H_1(M;\z)=\z$. Thus
an abelian representation $\r$ of $\pi_1(M)$
is determined by the matrix $\r(\m)$.
Now take a diagonal representation
$\r$ of $\pi_1(M)$ and assume that
$\r(\m)=\left(\begin{array}{cc}
 \sm&0\\0&\sm^{-1}\end{array}\right)$ such that
 $\sm\ne \pm 1$ and $\sm^2$ is not a root of the
 Alexander polynomial of $K$.
 As $dim\; X_0=1$,
 we just need to show, by Lemma \ref{dim equailty imples reduceness}, that for the diagonal representation $\r$ given above,
 we have  $dim\; T_{\chi_\r}(\fX(M))=1$.

It is shown in the proof of \cite[Lemma 4.8]{HP} that
 $H^1(\pi_1(M), sl_2(\c)_\r)=H^1(\pi_1(M),\c_0)$
where $\c_0=\c\left(\begin{array}{cc}
 1&0\\0&-1\end{array}\right)$ is a trivial $\pi_1(M)$-module.
Hence  $dim\; H^1(\pi_1(M), sl_2(\c)_\r)=1$.
For the given diagonal representation $\r$, $B^1(\pi_1(M), sl_2(\c)_\r)$ is two dimensional. Hence we have
 $dim\; Z^1(\pi_1(M), sl_2(\c)_\r)=3$, which
 implies that the representation $\r$ is a smooth point of $\fR(M)$
 since the component $R_0=\tr^{-1}(X_0)$ is of dimension $3$.

Now by \cite[Theorem 53 (3)]{Si}, we have
$$dim\; T_{\chi_\r}(\fX(M))=dim\; T_0(H^1(\pi_1(M), sl_2(\c)_\r)//S_\r)$$
where  $S_\rho$ is, in our current case,  the group of diagonal matrices and it
 acts  on $H^1(\pi_1(M), sl_2(\c)_\r$, in our current case, trivially
(as  the cohomology $H^1(\pi_1(M), sl_2(\c)_\r)=H^1(\pi_1(M),\c_0)$ is realized by cocycles taking values in diagonal matrices).
 Thus $dim\; T_0(H^1(\pi_1(M), sl_2(\c)_\r)//S_\r)=dim\; H^1(\pi_1(M),\c_0)=1$.
\qed

Combining Propositions \ref{canonical comps are reduced}
and \ref{abelian comp is reduced}, we have

\begin{cor}\label{fX is essentially reduced}Let $M$ be the exterior of a hyperbolic knot in $S^3$ such that
$X^\irr(M)=X^{\rg}(M)$. Then $\fX(M)$ is essentially reduced.
\end{cor}

\section{Proof of  Theorem \ref{AJ} -- a reduction}\label{section-AJ conjecture}

In this section we   briefly review
some background material and  give an outline of the
approach taken in
\cite{LT}, from which we can specify  the issues that we need to deal with in order to extend \cite[Theorem 1]{LT} to our current theorem, that is,
we reduce   Theorem \ref{AJ}
to  Proposition  \ref{fs is reduced}.

\subsection{Recurrence polynomial}
For a knot $K$ in $S^3$, let $J_{\sk, n}(t)\in\BZ[t^{\pm 1}]$ denote the \textit{$n$-colored Jones polynomial} of $K$
with the zero framing, which is the $sl_2$-quantum invariant of the knot colored by the $n$-dimensional representations \cite{RT}. We use the normalization  so that for the unknot $U$,
\[J_{\su,n}(t) = \frac{t^{2n}-t^{-2n}}{t^2-t^{-2}}.\]
By defining $J_{\sk,-n}(t) := -J_{\sk,n}(t)$ and $J_{K,0}=0$,
 one may  treat $J_{\sk,n}(t)$ as
 a discrete function  $$J_{\sk,-}(t): \z \to \z[t^{\pm 1}].$$

 \def\Ct{\BC[t^{\pm 1}]}
 The quantum torus
\[ \T = \c[t^{\pm 1}] \left<\sm^{\pm 1}, \sl^{\pm 1} \right> / (\sl \sm-t^2\sm \sl)\]
acts on   the set of all functions $f: \BZ \to \Ct$   by
\[ \sm f := t^{2n}f , \quad  (\sl f)(t) := f(n+1).\]
Now the set
  $${\mathcal A}_\sk := \{\alpha \in \T \mid \alpha J_{\sk,n}(t) = 0\},$$
  is obviously  a left ideal of $\T$, called the \textit{recurrence ideal} of $K$.
  By  \cite{GaLe} ${\mathcal A}_\sk$
  is not the zero ideal for every knot $K$ in $S^3$.
The ring $\T$ can be extended to a principal left ideal domain $\wT$ by adding inverses of all polynomials in $t$ and $\sm$.
The extended left ideal $\wA_\sk:=\wT {\mathcal A_\sk}$ is then generated by a single  nonzero polynomial in $\wT$, which  can be chosen to be of the form
\[\alpha_\sk(t, \sm,\sl) = \sum_{i=0}^m a_i(t,\sm) \sl^i,\]
with smallest total degrees in $t,\sm,\sl$ and with $a_0(t, \sm),...,a_m(t,\sm) \in \z[t,\sm]$ being coprime in $\z[t,\sm]$. The polynomial $\alpha_\sk(t, \sm,\sl)$ is uniquely determined up to a sign and is called  {\it the recurrence polynomial of $K$}.
When the framing of $K$ is 0, then $J_{K,n}(t) \in t^{2n-2}\BZ[t^{\pm 4}]$ (see eg. \cite{Le:Int}, with our $t$ equal to $q^{1/4}$ there). From here, it is not difficult to show that $\alpha_\sk(t, \sm,\sl)$ has only even powers in $t$ and even powers in $\sm$, i.e.
$a_i(t,\sm) \in \BZ[t^2,\sm^2]$ (see \cite[Proposition 5.6]{Le2}. It follows that
$\alpha_\sk(1,\sm,\sl)= \alpha_\sk(-1,\sm,\sl)$.

Now the AJ-conjecture asserts that for every knot $K$ in $S^3$,
   $\alpha_\sk(\pm 1,\sm,\sl)$
    is equal to the A-polynomial  of $K$, up to a factor of a polynomial in $\sm$, see \cite{G} and also \cite{FGL,Le,LT,Le2}.

\subsection{Kauffman bracket skein module}
For an oriented $3$-manifold $W$, we let $\mathcal S(W)$ denote the Kauffman bracket
skein module of $W$ over $\c[t^{\pm1}]$, which is the quotient
module of the free  $\c[t^{\pm1}]$-module generated by the set of
isotopy classes of framed links in $W$ modulo the well known Kauffman skein relations, see e.g. \cite{PS,Le,LT}.
A fundamental fact  is that when $\mathcal S(W)$ is specialized at $t=-1$ (which we denote by $\fs(W)$, i.e. $\fs(W)=\mathcal S(W)/(t+1)$), it acquires a ring structure and is naturally isomorphic as a ring to the universal character ring of $\pi_1(W)$, i.e.
$$\fs(W)=\fB(W).$$ So $\fs(W)/\sqrt{(0)}$ is isomorphic to the coordinate ring of
$X(W)$ (see \cite{Bu} \cite{PS}).
For the exterior $M$ of a knot $K$ in $S^3$, we shall simply write $\mathcal S$ for $\mathcal S(M)$ and $\fs$ for $\fs(M)$.

If $F$ is an oriented surface, we define $\mathcal S(F):=\mathcal S(F\times [0,1])$.
Then $\mathcal S(F)$ has a natural algebra structure, where the product of two framed links $L_1, L_2$ is obtained by placing $L_1$ atop $L_2$.
For a torus $T^2$, $\mathcal S(T^2)$ can be identified, as an $\c[t^{\pm1}]$-algebra,  with
$$\T^\s:=\{f\in\T; \s(f)=f\}$$ where $\s:\T\ra \T$ is the involution defined by $\s(\sm)=\sm^{-1}$
and $\s(\sl)=\sl^{-1}$ (see \cite{FG}).

If $M$ is the exterior of knot $K$ in $S^3$,
there is a natural map
\begin{equation}\label{eq:Theta map}
\Theta: \mathcal S(\p M)=\T^\s\ra \mathcal S=\mathcal S(M)
\end{equation}
induced by the inclusion $\p M\hookrightarrow M$.
Then  $\P:=ker (\Theta)$ is called the {\it quantum peripheral ideal of $K$} and by \cite{FGL} and \cite{G2}, $\P\subset \A_\sk$ (see also \cite[Corollary 1.2]{LT}).

\subsection{Dual contruction of $A$-polynomial}
On the other hand, there is a dual construction of the $A$-polynomial of a knot $K$ in $S^3$.
Let $\ft:=\c[\sm^{\pm1},\sl^{\pm1}]$, which is the function ring of $\BCxt$,
and let $\ft^\s:=\{f\in \ft; \s(f)=f\}$, which is the function ring of $X(\p M)$.
%,where $\s:\ft\ra\ft$ is the involution defined by $\s(\sm)=\sm^{-1}, \s(\sl)=\sl^{-1}$.
The restriction map $\iota^*: X(M)\ra X(\p M)$ induces
a ring homomorphism between coordinate rings
\begin{equation}\label{eq:theta map}\theta: \c[X(\p M)]=\ft^\s\ra \c[X(M)].
\end{equation}
Let $\fp:=ker (\theta)$, which is called {\it the classical peripheral ideal of the knot $K$}.
Now extend $\ft$ naturally to the principal ideal domain $\tilde\ft:=\c(\sm)[\sl^{\pm1}]$
where $\c(\sm)$ is the fractional field of $\c[\sm]$.
Then the extended ideal $\tilde \fp:=\tilde \ft \fp$
of $\fp$ in $\tilde \ft$ is generated by a single polynomial which can be
normalized to be of the form
\[B_\sk(\sm,\sl) = \sum_{i=0}^m b_i(\sm) \sl^i,\]
with smallest total degree and with $b_0(\sm),...,b_m(\sm) \in \z[\sm]$ being coprime in $\z[\sm]$.
So  $B_\sk(\sm,\sl)$  is uniquely defined up to a sign.
The polynomial $B_\sk(\sm,\sl)$ is  called {\it the $B$-polynomial of $K$}
and is equal to the $A$-polynomial $A_\sk(\sm,\sl)$ divided by its
$\sm$-factor (see \cite[Corollary 2.3]{LT}).

Note that the universal character ring of $\p M$ is reduced,
so we have $\fs(\p M)=\c[\p M]=\ft^\s$. Specializing (\ref{eq:Theta map}) at $t=-1$, we get
\begin{equation}\label{eq:theta map again}
\theta: \ft^\s\ra \fs=\fs(M)
\end{equation}
in which the map $\theta$ is the same one given in (\ref{eq:theta map}).

\subsection{Localized skein module and reduction of Theorem \ref{AJ}}
Note that the inclusion map $\p M\subset M$ also
induces a left ${\mathcal S}(\p M)=\T^\s$-module structure on ${\mathcal S}={\mathcal S}(M)$.
Let $D:=\c[t^{\pm1},\sm^{\pm1}]$,
 $D^\s:=\{f\in D; \s(f)=f\}$ where $\s$ is the involution defined by
  $\s(\sm)=\sm^{-1}$ and $\overline D$ the localization
of $D$ at $(1+t)$, i.e.
$$\overline D:=\{f/g; f, g\in D, g\notin (1+t)D\}.$$
Then we may consider ${\mathcal S}$, as well as $\T^\s$,  as
a left $D^\s$-modules as $D^\s$ is contained in $\T^\s$.
Now let
$$(\overline \T\stackrel{\overline \Theta}{\lra} \overline{\mathcal S}):=(\T^\s\stackrel{\Theta}{\lra} \mathcal S)\otimes_{D^\s} \overline D, \quad
(\overline \ft\stackrel{\overline \theta}{\lra} \overline{\fs}):=(\ft^\s\stackrel{\theta}{\lra} \fs)
\otimes_{\c[\sm^{\pm1}]^\s}\c(\sm).$$
We shall consider $\overline{\mathcal S}$ as a left  $\overline D$-module
and call it {\it the localized skein module of $M$}.
The following commutative diagram:
$$\begin{diagram}[height=2em,w=3em]
\overline \T&\rTo^{\overline \Theta}&\overline{\mathcal S}\\
\dTo^{\e}&&\dTo_{\e}\\
 \overline\ft&\rTo^{\overline\theta}& \overline\fs
\end{diagram}$$
is obtained in \cite{LT} as Lemma 3.2, where the vertical maps 
are the natural projections $\mathcal M\ra \mathcal M/(t + 1)$, for
$\mathcal M =\overline\T$ and $\mathcal M= \overline{\mathcal S}$.
We claim  that the proof of Theorem \ref{AJ}
can be reduced to the proof of the following proposition:

\begin{prop}\label{fs is reduced}Let $M$ be the exterior of a hyperbolic knot in $S^3$. If
$X^{\rg}(M)=X^\irr(M)$ (or equivalently, $\oX^{\rg}(M)=\oX^\irr(M)$), and the two discrete faithful characters of $\overline X(M)$ lie in the same component of
$\overline X(M)$, then
\newline
(1) the ring $\overline\fs$ is  reduced, and
\newline
(2) the map $\bar\theta$ is surjective.
\end{prop}

Assuming Proposition \ref{fs is reduced}, we may
finish the proof of Theorem \ref{AJ} as follows.
By Condition (1) of Theorem \ref{AJ}, we have
Proposition \ref{fs is reduced}.
Combining Proposition \ref{fs is reduced} with
Condition (3) of Theorem \ref{AJ}, we may apply
 \cite[Corollary 3.6]{LT} to have $$\a_\sk(-1,\sm,\sl)|B_\sk(\sm,\sl) \in \BZ[\sm^2,\sl].$$
Condition (1) of Theorem \ref{AJ}
and Proposition \ref{bal-irred}
together imply
 $A_\sk(\sm,\sl)=(\sl-1)\widehat A_\sk(\sm,\sl)=B_\sk(\sm,\sl)$
 and $\widehat A_\sk(\sm,\sl)$ are balanced-irreducible in $ \BZ[\sm^2,\sl]$.
It's known that $L-1$ is a factor $\a_\sk(-1,\sm,\sl)$ (\cite[Proposition 2.3]{Le}).
By Condition (2) of Theorem \ref{AJ} and
 \cite[Lemma 3.9]{LT}
  we know that the $\sl$-degree of $\a_\sk(-1,\sm,\sl)$  is greater than or equal to 2.
  As $\a_\sk(-1,\sm,\sl)$ is also balanced (see the lemma below) and its coefficients are all integers, the polynomial
  $\widehat\a_\sk(-1,\sm,\sl):=\a_\sk(-1,\sm,\sl)/(\sl-1)$
  is also balanced, belongs to $\z[\sm^2,\sl]$, and has $\sl$-degree $\ge 1$.
 Since $\widehat\a_\sk(-1,\sm,\sl)$ divides $\widehat A_\sk(\sm,\sl)$ which is balanced-irreducible in $ \BZ[\sm^2,\sl]$, the two polynomials must be equal (up to sign).
 Hence $\a_\sk(-1,\sm,\sl)=A_\sk(\sm,\sl)$ (up to sign).

\begin{lemma}Let $\a_\sk(t,\sm,\sl)$ be the normalized
recurrence polynomial of a knot $K$ in $S^3$.
Then $\a_\sk(-1,\sm,\sl)$ is a balanced polynomial.
\end{lemma}

\pf
By \cite[Theorem 1.4]{G2},
the recurrence (left)  ideal $\A_\sk$ of $K$
 is invariant under the involution
 $\s$ of the quantum torus $\T$  defined by
 $\s(\sm)=\sm^{-1}, \s(\sl)=\sl^{-1}$.
  Hence $\s(\a_\sk(t,\sm,\sl))=\a_\sk(t,\sm^{-1},\sl^{-1})$ is contained in $\A_\sk$.
  Suppose the $\sl$-degree of $\a_\sk(t,\sm,\sl)$ is $m$.
  Then using the relation $\sl\sm=t^2\sm\sl$,
  one can easily see that there is a monomial
  $t^{2a}\sm^b\sl^m$, for some integers $a, b$, with $b\geq 0$, such that
  $t^{2a}\sm^b\sl^m\a_\sk(t,\sm^{-1},\sl^{-1})$ is contained in
  $\z[t,\sm,\sl]$ of $\sl$-degree $m$ with relatively
  prime coefficients with respect to the variable $\sl$.
  It follows that $t^{2a}\sm^b\sl^m\a_\sk(t,\sm^{-1},\sl^{-1})$
  is also a generator of $\tilde \A_\sk$ and by the unique
  normalized form of such generator, we have
  $$t^{2a}\sm^b\sl^m\a_\sk(t,\sm^{-1},\sl^{-1})=\a_\sk(t,\sm,\sl)$$
  up to  sign. Hence $\sm^b\sl^m\a_\sk(-1,\sm^{-1},\sl^{-1})=\a_\sk(-1,\sm,\sl)$
  up to sign, i.e. $\a_\sk(-1,\sm,\sl)$ is balanced.\qed

\section{Proof of Proposition \ref{fs is reduced}}

Under the assumptions of Proposition \ref{fs is reduced},
we  know, by Corollary \ref{fX is essentially reduced}, that the character scheme $\fX(M)$ is essentially reduced, i.e.
the universal character ring $\fB(M)$ is essentially reduced.
We may assume that $$\fB(M)=\c[x_1,...,x_n]/I$$
where $I$ is an ideal $\c[x_1,...,x_n]$.
We may also assume that the ideal $I$ has an irredundant primary decomposition
$$I=\bigcap_{j=0}^{m}Q_j$$
such that $Q_0,Q_1,...,Q_k$ are  the isolated components of $I$,
$Q_{k+1},...,Q_m$ embedded components,
with $Q_0$ defining  the abelian component $X_0$ of $X(M)$, $Q_1,...,Q_k$ defining
the components $X_1,...,X_k$ of $X^{\rg}(M)$ respectively.
We have that $Q_0,Q_1,...,Q_k$ are prime ideals.
As $X_0,X_1,...,X_k$ are all $1$-dimensional,
the zero locus of each $Q_j, j=k+1,...,m$,
is a point, and thus $\sqrt{Q_j}=P_j$ is a maximal ideal, i.e. for some point $(a_1,...,a_n)$ in $X_0\cup X_1\cup...\cup X_k$, $P_j=(x_1-a_1,...,x_n-a_n)$.

Let $R_n=\c[x_1,...,x_n]$. Then $\fB(M)=R_n/I$.
We may assume that the coordinate $x=x_1$ in
$\fB(M)=\c[x_1,...,x_n]/I$
represents the function
$x:\fX(M)\ra \c$ given by
$x(\chi_\r)=trace(\r(\m))$ where $\m$ is a meridian of $\pi_1(M)$.
Let $S = \c[x] \setminus \{0\}$, which is a multiplicative subset of $\c[x]$. If $\mathcal M$ is a $\c[x]$-module,
let $S^{-1}\mathcal M$ denote the  localization of $\mathcal M$ with respect to  $S$.
Note that every ideal $J$ in $R_n$ is a $\c[x]$-module and so is $R_n/J$.

\begin{lemma}
 For $j> k$ we have
$$  S^{-1}Q_j=S^{-1}R_n.$$
\end{lemma}

\begin{proof}
For $j>k$, $P_j=(x_1-a_1,\dots, x_n - a_n)$
 is a maximal ideal.
Since $Q_j$ is primary and  $\sqrt {Q_j}=P_j$, we have $P_j^d  \subset Q_j$ for some integer $d>0$. It follows that $(x-a_1)^d \in Q_j$. Since
$(x-a_1)^d \in S$, we have $1 \in S^{-1} Q_j$. Hence $S^{-1} Q_j = S^{-1} R_n$.
\end{proof}

Hence $S^{-1}I=\bigcap_{j=0}^mS^{-1}Q_j
=\bigcap_{j=0}^kS^{-1}Q_j$.
As $Q_j$, $j<k$, are prime, $S^{-1}I=
\bigcap_{j=0}^kS^{-1}Q_j$ is a prime decomposition of
the ideal $S^{-1}I$ in $S^{-1}R_n$. Therefore
$S^{-1}(R_n/I)=S^{-1}R_n/S^{-1}I$ is a reduced ring.
By definition,
$$S^{-1}(R_n/I)=\fB(M)\otimes_{\c[x]}\c(x)=\fs\otimes_{\c[x]}\c(x)
=\fs\otimes_{\c[\sm+\sm^{-1}]}\c(\sm+\sm^{-1}).$$
Taking tensor product of this with $\c(\sm)$, we have
 $$(\fs\otimes_{\c[x]}\c(x))\otimes_{\c(x)}\c(\sm)
=\fs\otimes_{\c[\sm^{\pm 1}]^\s}\c(\sm)=\bar\fs$$
which is still reduced.
This proves part (1) of  Proposition \ref{fs is reduced}.

From the above proof, we also get
$$\fs\otimes_{\c[x]}\c(x)=\c[X(M)]\otimes_{\c[x]}\c(x)$$
because $\c[X(M)]=R_n/I'$ with $I'=\bigcap_{j=0}^kQ_j$
and $S^{-1}I=S^{-1}I'$.
 The restriction of the function $x$  on $X_1$
 is nonconstant and thus is non-constant
on $X_j$  for each $j=1,...,k$.
 It is easy to see  that $x$ is also non-constant on $X_0$.
 Hence  a similar proof as that of \cite[Lemma  3.8]{LT}
 shows that
$$\c[X(M)]\otimes_{\c[x]}\c(x)
=\bigoplus_{j=0}^k \c[X_j]\otimes_{\c[x]}\c(x).$$
Note that $\c[X_j]\otimes_{\c[x]}\c(x)$ is isomorphic to the
field of rational functions on $X_j$ for each $j=0,1...,k$,
(by \cite[Lemma 3.7]{LT}).

Recall that $\iota^*: X(M)\to X(\partial M)$ is the restriction map which induces the ring homomorphism
$\theta: \BC[X(\partial M)] \to \BC[X(M)]$.
Also recall that $Y_j$  is the Zariski closure of $\iota^*(X_j)$ in $X(\p M)$, $j=0,1,...,k$.
As $x$ is non-constant on each $Y_j$, $\c[Y_j]\otimes_{\c[x]}\c(x)$ is isomorphic to the
field of rational functions on $Y_j$ for each $j=0,1...,k$.
By Theorem \ref{distinct3}, $\iota^*:X_j\ra Y_j$ is
a birational isomorphism for each $j=1,...,k$.
When $j=0$, $\iota^*:X_0\ra Y_0$ is also a birational isomorphism,
which is an elementary fact.
 Hence the map $\iota^*$  induces an isomorphism
$$\c[Y_j]\otimes_{\c[x]}\c(x)\ra \c[X_j]\otimes_{\c[x]}\c(x)$$
for each $j=0,1,...,k$.
As $Y_j$, $j=0,1,...,k$, are distinct curves in $X(\p M)$ by Theorem \ref{distinct3},
$\iota^*$ induces the isomorphism
$$\bigoplus_{j=0}^{k}\c[Y_j]\otimes_{\c[x]}\c(x)\ra \bigoplus_{j=0}^{k}\c[X_j]\otimes_{\c[x]}\c(x)=\c[X(M)]\otimes_{\c[x]}\c(x)$$
which implies that the map
$$\c[X(\p M)]\otimes_{\c[x]}\c(x)\ra \c[X(M)]\otimes_{\c[x]}\c(x)$$
induced by $\iota^*$ is surjective since $Y_0\cup Y_1\cup\cdots\cup Y_k$ is a subvariety of
$X(\p M)$.
Taking tensor product of this map with
$\c(\sm)$ over $\c(x)$ and noting that $\c[X(\p M)]=\ft^\s$
and $\c[X(M)]\otimes_{\c[x]}\c(x)=\fs\otimes_{\c[x]}\c(x)$,
we get the map
\begin{equation}\label{eq:surjective}
\ft^\s\otimes_{\c[\sm^{\pm1}]^\s}\c(\sm)\ra \fs\otimes_{\c[\sm^{\pm1}]^\s}\c(\sm)\end{equation}
which is still surjective.
Now  one can check that (\ref{eq:surjective}) is
precisely the map
$$\overline\theta:\overline\ft\ra\overline\fs.$$
Part (2) of  Proposition \ref{fs is reduced} is proved.

\section{Proof of Theorem \ref{2-bridge}}

Let $M$ be the exterior of a hyperbolic $2$-bridge knot in $S^3$.
We call a character  $\bar\chi_{\bar\r}\in \overline X(M)$ (resp.  $\chi_\r\in X(M)$) {\it dihedral}
if it is the character of a dihedral representation
 i.e. a representation whose image is
 a dihedral group (resp. a binary dihedral group).
It was shown in \cite[Section 5.3]{T}   (see also \cite[Appendix A]{BB}) that
any dihedral character of $\overline X(M)$ (and of $X(M)$)  is a smooth point
and thus is contained in a unique $\c$-component of $\overline X(M)$ (resp. $X(M)$).
It was also shown in \cite[Section 5.3]{T} that every
$\c$-component of $\overline X^\irr(M)$ contains a dihedral character.

Since  a dihedral character of $\overline X^\irr(M)$ is  real valued, it is  a fixed point of the $\t$-action (the complex conjugation action given in Subsection \ref{sec.prelim}). It follows that every $\c$-component of $\overline X^\irr(M)$ is invariant under the $\t$-action.
Hence in particular  the two discrete faithful characters of $\overline X(M)$ are contained in the same $\c$-component of $\overline X(M)$.

By \cite[Lemma 5.5 (3)]{BZ3}, any dihedral character in $X(M)$ is a fixed point of the $\e$-action (recall its definition in Subsection \ref{e-action}).
It follows that every $\c$-component of $X^\irr(M)$ is invariant under the $\e$-action, as well as the $\t$-action, which implies
that all the four discrete faithful characters of $X(M)$ are contained in the same $\c$-component of $X(M)$, say $X_1$. Therefore
$X^{\rg}(M)$ is the $\Aut(\c)$-orbit of $X_1$ and thus is  $\q$-irreducible.

%%%%%%%%%%%%%%%%%%%%%%%%%%%%%%%%%%%%%%%%%%%
\def\bysame{$\underline{\hskip.5truein}$}
%%%%%%%%%%%%%%%%%%%%%%%%%%%%%%%%%%%%%%%%%%%

\end{document}